\documentclass[11pt]{amsart}

\usepackage[T1]{fontenc}
\usepackage{mlmodern}
\usepackage{microtype}

\usepackage{amsmath,amsthm,amsfonts,amssymb,latexsym,textcomp}
\usepackage{mathtools}   
\usepackage{mathrsfs}    
\usepackage{bm}

\usepackage{graphicx}
\usepackage{xcolor}
\usepackage{caption}
\usepackage{subcaption}
\usepackage{enumerate}
\usepackage{float}
\usepackage[a4paper, hmargin=2.7cm, vmargin=2.3cm]{geometry}
\usepackage{tikz}
\usepackage{tikz-cd}
\usetikzlibrary{calc,decorations.pathreplacing}
\tikzcdset{row sep/Large/.initial=5em}
\tikzcdset{column sep/Large/.initial=5em}

\definecolor{theorem_color}{rgb}{0.35,0.0,0}
\definecolor{VividBurgundy}{RGB}{159,29,53}
\definecolor{Burgundy1}{RGB}{128,0,32}
\definecolor{darkbrown}{rgb}{0.4, 0.26, 0.13}
\definecolor{maroon}{rgb}{0.5, 0.0, 0.0}
\definecolor{gold}{rgb}{0.85, 0.65, 0.13}
\definecolor{BurntUmber}{RGB}{138, 51, 36}
\definecolor{darkpink}{rgb}{0.761,0.094,0.357}
\definecolor{darkorange}{HTML}{CC5500} 
\colorlet{LightBurntUmber}{BurntUmber!10!white}
\colorlet{figaccent}{darkorange} 

\numberwithin{equation}{section}

\makeatletter
\let\chi@orig\chi
\newcommand{\lifted@chi}[2]{\setbox\@tempboxa\hbox{$\m@th#1\chi@orig$}\mathord{\raise\dp\@tempboxa\box\@tempboxa}}
\newcommand{\liftedchi}{\mathpalette\lifted@chi\relax}
\renewcommand{\chi}{\liftedchi}
\makeatother

\newif\ifsubmission
\submissionfalse   
\ifsubmission
  \newcommand{\todo}[1]{}%
  \newcommand{\wap}{\errmessage{Missing proof (\string\wap) in submission build: resolve or remove it before submitting}}%
\else
  \newcommand{\todo}[1]{\textcolor{red}{#1}}%
  \newcommand{\wap}{\todo{Write a Proof!}}%
\fi

\theoremstyle{plain}
\newtheorem{theorem}{Theorem}[section]
\newtheorem{lemma}[theorem]{Lemma}
\newtheorem{proposition}[theorem]{Proposition}
\newtheorem{corollary}[theorem]{Corollary}

\theoremstyle{definition}
\newtheorem{definition}{Definition}[section]
\newtheorem{remark}[theorem]{Remark}

\theoremstyle{plain}
\newtheorem{conjecture}{Conjecture}

\DeclarePairedDelimiter{\prn}{(}{)}
\DeclarePairedDelimiter{\brk}{[}{]}
\DeclarePairedDelimiter{\abs}{|}{|}
\DeclarePairedDelimiter{\nrm}{\|}{\|}
\DeclarePairedDelimiter{\setbr}{\{}{\}}
\DeclarePairedDelimiter{\inn}{\langle}{\rangle}
\DeclarePairedDelimiter{\flr}{\lfloor}{\rfloor}

\newcommand{\lrp}[1]{\prn*{#1}}
\newcommand{\lrb}[1]{\brk*{#1}}

\newcommand{\lrset}[1]{\setbr*{#1}}
\newcommand{\lrab}[1]{\inn*{#1}}

\newcommand{\ab}[1]{\inn{#1}}
\newcommand{\norm}[1]{\nrm{#1}}
\newcommand{\set}[1]{\setbr{#1}}
\newcommand{\floor}[1]{\flr{#1}}

\newcommand{\define}{\textit}

\newcommand{\mc}{\mathcal}

\newcommand{\mr}{\mathscr}

\newcommand{\one}{\mathbf{1}}

\newcommand{\vp}{\varphi}

\newcommand{\suppressthis}[1]{}

\newcommand{\C}{\mathbb{C}}

\newcommand{\Q}{\mathbb{Q}}
\newcommand{\R}{\mathbb{R}}
\newcommand{\T}{\mathbb{T}}
\newcommand{\Z}{\mathbb{Z}}

\DeclareMathOperator{\supp}{supp}

\DeclareMathOperator{\Image}{Image}

\usepackage{hyperref}
\hypersetup{
    colorlinks=true,
    citecolor=darkpink,   
    linkcolor=darkorange, 
    urlcolor=darkpink,    
}

\begin{document}

\title[A Nivat counterexample for a full-affine-span window]{A Counterexample to Nivat's Conjecture for a Non-Convex Window of Full Affine Span}
\author{Abhishek Khetan}
\email{abhishek.khetan@ashoka.edu.in}
\date{\today}

\begin{abstract}
	We construct an exact cluster $F\subseteq\Z^2$ of cardinality $8$ with full affine span, together with an $F$-tiling $T$, such that the orbit closure of $T$ in $\set{0, 1}^{\Z^2}$ does not contain a $1$-periodic $F$-tiling.
	Since every $F$-tiling is a low-complexity configuration with respect to the window $\bar F := \set{-a: a\in F}$, this supplies a ``non-degenerate'' counterexample, in a strong sense, to Nivat's conjecture for non-convex windows.
	This answers, in the negative, a question of Kari and Moutot (2023) whether every such counterexample must be degenerate, in the sense that the probing window is contained in a coset of a proper finite-index sublattice.

	We complement this with a positive result: for every exact cluster $F$ of full affine span whose cardinality is the square of a prime, every $F$-tiling has a $1$-periodic $F$-tiling in its orbit closure.
	Together with Szegedy's theorem that every tiling by a cluster of prime cardinality is $1$-periodic, this shows that no cluster of fewer than $8$ cells can exhibit the phenomenon, with the possible exception of cardinality $6$, which we leave open.
\end{abstract}

\maketitle

\tableofcontents


\section{Introduction}
\label{section:introduction}

\subsection{Low complexity and Nivat's conjecture}

A \emph{configuration} is a colouring $c\colon\Z^2\to\mc A$ of the integer plane by finitely many symbols.
One measures how complicated a configuration is by counting the local patterns it displays through a fixed finite viewing region.

\begin{definition}[Window and complexity]
	\label{definition:complexity}
	A \define{window} (or \define{shape}) is a finite subset $D\subseteq\Z^2$.
	For a configuration $c\colon\Z^2\to\mc A$, the \define{$D$-patterns} of $c$ are the restrictions $c|_{v+D}$ as $v$ ranges over $\Z^2$ (the snapshots seen through $D$ as it slides over the plane), and the \define{$D$-complexity} $P_c(D)$ is the number of distinct $D$-patterns.
	The configuration $c$ has \define{low complexity} with respect to $D$ if
	\[
		P_c(D)\;\le\;\abs D,
	\]
	so that the window sees no more patterns than it has cells.
\end{definition}

A constant configuration has $P_c(D)=1$ for every $D$, while a generic one has $P_c(D)$ as large as
$\abs{\mc A}^{\abs D}$. Low complexity is thus a strong constraint, and configurations meeting it are
expected to be rigid. To make ``rigid'' precise we record the shift and its invariants.

\begin{definition}[Periodicity]
	\label{definition:periodicity}
	A configuration $c$ is \define{$1$-periodic} if there is a nonzero $v\in \Z^2$ such that $c(x+v)=c(x)$ for all $x\in \Z^2$.
	The \define{period group} of $c$ is the subgroup
	\[
		\set{v\in\Z^2 : c(x+v)=c(x) \text{ for all } x\in \Z^2}.
	\]
	Thus $c$ is $1$-periodic if and only if the period group has rank at least 1.
	The configuration is \define{biperiodic} if its period group has finite index in $\Z^2$, and \define{rank-$0$} (it \define{has no period}) if its period group is trivial.
\end{definition}

The set $\mc A^{\Z^2}$ of all configurations, with the product topology, is a compact metrizable space, and $\Z^2$ acts on it by the \define{shift} $(v\cdot c)(x)=c(x-v)$ for all $x\in \Z^2$.

\begin{definition}[Orbit closure]
	\label{definition:orbit closure}
	The \define{orbit closure} of a configuration $c$ is
	\[
		\overline{\Z^2\cdot c}=\operatorname{cl}\set{v\cdot c : v\in\Z^2},
	\]
	the smallest closed shift-invariant set containing $c$.
\end{definition}

The two definitions above measure complexity at opposite scales.
The count $P_c(D)$ is a \emph{local} statistic: it records only what the configuration reveals through one bounded window.
Periodicity is a constraint of the opposite kind: a period $v$ constrains the values of $c$ at points arbitrarily far apart.
In this language the bound $P_c(D)\le\abs D$ limits the \emph{local} complexity, while periodicity is a form of low \emph{global} complexity.
The basic prediction relating low complexity to periodicity is Nivat's conjecture, raised by Nivat in his $1997$ \textsc{icalp} keynote \cite{nivat1997}.
It is the two-dimensional analogue of the Morse--Hedlund theorem \cite{morse_hedlund_1940}.

\begin{conjecture}[Nivat]
	\label{conjecture:nivat}
	Let $D\subseteq\Z^2$ be a rectangle and $c\colon\Z^2\to\mc A$ a configuration with
	$P_c(D)\le\abs D$. Then $c$ is $1$-periodic.
\end{conjecture}

The conjecture remains open, but strong partial results are known.
Cyr and Kra \cite{cyr_kra_nonexpansive} proved that the conclusion holds under the stronger hypothesis $P_c(D)\le\tfrac12\abs D$ for a rectangle $D$.
Kari and Szabados \cite{kari_szabados_alg_geom}, by an algebraic-geometric method, established under the threshold hypothesis $P_c(D)\le\abs D$ a structural decomposition of low-complexity configurations as sums of periodic ones, and deduced an asymptotic form of the conjecture: a non-periodic configuration can be of low complexity with respect to only finitely many rectangles.
The convex analogue, that $P_c(D)\le\abs D$ for a finite \emph{convex} window $D$ (the restriction to $\Z^2$ of a convex region of $\R^2$) forces $1$-periodicity, was raised by Sander and Tijdeman \cite{sander_tijdeman_lattices}.
For arbitrary convex windows in place of rectangles, Kari and Moutot \cite{kari_mutot_low_comp} established the orbit-closure form: if $P_c(D)\le\abs D$ for a finite convex $D\subseteq\Z^2$, then $\overline{\Z^2\cdot c}$ contains a $1$-periodic configuration.
Whether $c$ itself must be $1$-periodic is the still-open convex form of the conjecture.

It is then natural to ask whether the rectangle, or the convexity relaxation, is essential: \emph{does low complexity $P_c(D)\le\abs D$ force periodicity for an \emph{arbitrary} finite window $D$?} Here the answer is no. 
As Kari and Moutot observe \cite{kari_mutot_low_comp}, the conclusion fails for general windows, but ``all counterexamples we know are based on periodic sublattices'': the window is contained in a coset of a proper finite-index sublattice.
The standard construction, going back to Cassaigne \cite{cassaigne_subword_complexity} (see also \cite{kari_low_complexity_survey, kari_mutot_low_comp}), superimposes two independent periodic layers indexed by the parity of the first coordinate.
Call a cell $(x,y)\in\Z^2$ \emph{even} or \emph{odd} according as $x$ is even or odd, so that $\Z^2$ splits into its even and odd columns.
On the even cells one writes a horizontally periodic colouring and on the odd cells a vertically periodic one, and one reads the result through a window all of whose cells are even, so that each placement of the window falls inside a single parity class.
Such a window never sees the two layers at once, so it cannot detect that their periods disagree, and the combined configuration
has low complexity yet no period.
In every such construction the window has \emph{proper} affine span.
They raise as a direction for future study the question of whether this sublattice structure is the only obstruction.
Before we proceed, we record the following terminology.

\begin{definition}[Affine span]
	\label{definition:affine span}
	A finite set $D\subseteq\Z^2$ has \define{full affine span} if its difference set $D-D=\set{a-b : a,b\in D}$ generates $\Z^2$ as a group.
	Equivalently, $D$ has full affine span if the smallest affine sublattice of $\Z^2$ containing $D$ is $\Z^2$ itself.
\end{definition}

Thus finding a non-degenerate counterexample is the same as asking for a window $D$ with full affine span together with a configuration $c\colon\Z^2\to\mc A$ that has low complexity with respect to $D$ but is not $1$-periodic.
This article has two complementary results.

\begin{enumerate}[(1)]
	\item \emph{A non-degenerate counterexample, in a strong sense.}
	We exhibit a window of full affine span and a configuration $c$ of low complexity with respect to it for which not even the orbit closure of $c$ contains a $1$-periodic configuration.

	\item \emph{A positive periodicity result.}
		The configurations we construct arise from tilings (see \S\ref{subsection:tilings intro}) of $\Z^2$ by a finite tile $F$.
		We show that whenever $\abs F$ is the square of a prime, every tiling of $\Z^2$ by translates of $F$ has a $1$-periodic $F$-tiling in its orbit closure.
\end{enumerate}

\subsection{Tilings as low-complexity configurations}
\label{subsection:tilings intro}
The configurations we use come from tilings.

\begin{definition}[Cluster]
	\label{definition:cluster}
	A finite subset of $\Z^2$ is a \define{cluster} (or a \define{tile}). 
\end{definition}

\begin{definition}[Tiling and exactness]
	\label{definition:tiling}
	An $F$-\define{tiling} is a subset $T\subseteq\Z^2$ such that every
	element of $\Z^2$ has a unique representation $a+t$ with $a\in F$ and $t\in T$.
	We write this as
	$\Z^2=F\oplus T$. The elements of $T$ will be referred to as \define{anchors}: each $t\in T$ marks one placed copy
	$F+t$ of the tile. A cluster $F$ is \define{exact} if it admits an $F$-tiling of $\Z^2$.
\end{definition}

Equivalently, $T$ is an $F$-tiling if and only if
\begin{equation}
	\label{equation:tiling condition}
	\sum_{a\in F} 1_T(v-a) = 1 \qquad\text{for all } v\in\Z^2,
\end{equation}
that is, every point of $\Z^2$ is covered by exactly one translate $F+t$.
Encoding $T$ by its indicator $1_T\in\set{0,1}^{\Z^2}$ turns a tiling into a configuration over the two-symbol alphabet.

For each fixed $v\in\Z^2$ the constraint \eqref{equation:tiling condition} reads $\sum_{a\in F}x(v-a)=1$ and involves only the finitely many coordinates $x(v-a)$, $a\in F$, so it cuts out a clopen subset of $\set{0,1}^{\Z^2}$.
The set of all $F$-tilings,
\[
	X_F=\lrset{x\in\set{0,1}^{\Z^2} : \sum_{a\in F}x(v-a)=1 \text{ for all } v\in\Z^2},
\]
is the intersection of these subsets over $v\in\Z^2$, hence closed, and it is clearly shift-invariant.
Thus $X_F$ is a subshift, and the orbit closure $\overline{\Z^2\cdot 1_T}$ of any $F$-tiling is contained in it.
In particular every configuration in $\overline{\Z^2\cdot 1_T}$ is again the indicator of an $F$-tiling.
We use this fact repeatedly, as it is what makes ``a $1$-periodic point of the orbit closure'' the same as ``a $1$-periodic $F$-tiling in the orbit closure.''

\subsubsection{Tilings are low-complexity.}
	\label{point:low complexity dictionary}
	Let
	\[
		\bar F=-F=\set{-a : a\in F}
	\]
	be the \define{reflected} cluster, and regard an $F$-tiling $T$ as the configuration $1_T\in\set{0,1}^{\Z^2}$.
	Rewriting \eqref{equation:tiling condition} as
	\[
		\abs{T\cap(v+\bar F)} 
		= |\set{(t,a)\in T\times F : t+a=v}|
		= 1 \qquad\text{for all } v\in\Z^2,
	\]
	we see that every translate $v+\bar F$ of the reflected cluster contains exactly one anchor of $T$.
	Hence the pattern $1_T|_{v+\bar F}$ is always a single $1$ among the $\abs F$ cells of $\bar F$, so
	\[
		P_{1_T}(\bar F) \;\leq\; \abs{\bar F} \;=\; \abs F .
	\]
	Every $F$-tiling is therefore a low-complexity configuration for the window $\bar F$ (which is in general non-convex).
To produce the counterexample promised above, it thus suffices to exhibit an exact cluster $F$ of full
affine span together with an $F$-tiling $T$ none of whose orbit-closure limits is periodic.
Therefore it suffices to refute the following conjecture.

\begin{conjecture}
	\label{conjecture:orbit closure one periodicity}
	Let $F$ be an exact cluster with full affine span. Then for every $F$-tiling $T$, the orbit
	closure $\overline{\Z^2\cdot T}$ contains a $1$-periodic $F$-tiling.
\end{conjecture}

\subsection{Results}

Our main result refutes Conjecture~\ref{conjecture:orbit closure one periodicity}, and with it the
non-convex analogue of Nivat's conjecture, in the strongest topological sense and with a
non-degenerate window.

\begin{theorem}[\S\ref{section:counterexample}]
	\label{theorem:main counterexample}
	There exist an exact cluster $F\subseteq\Z^2$ with $\abs{F}=8$ and full affine span, and an $F$-tiling $T$, such that the orbit closure $\overline{\Z^2\cdot T}$ contains no $1$-periodic $F$-tiling.
	Consequently, the configuration $1_T$ has complexity $\le\abs{\bar F}$ with respect to the non-convex, full-affine-span window $\bar F$, yet no configuration in its orbit closure is $1$-periodic.
\end{theorem}

The mechanism of two transverse families of one-dimensional Sturmian ``bricks'' that never share a period is elementary in contrast to the intricate aperiodic clusters of Greenfeld and Tao \cite{greenfeld_tao_counterexample}.
We complement the counterexample with a positive result on the tilings from which our configurations
arise.

\begin{theorem}[\S\ref{section:prime squared}]
	\label{theorem:p squared intro}
	Let $F\subseteq\Z^2$ be an exact cluster of full affine span with $\abs F=p^2$ for a prime $p$.
	Then for every $F$-tiling $T$, the orbit closure $\overline{\Z^2\cdot T}$ contains a $1$-periodic $F$-tiling.
	Equivalently, Conjecture~\ref{conjecture:orbit closure one periodicity} holds when $\abs F=p^2$.
\end{theorem}

It was already shown by Szegedy \cite{szegedy_algorithms_to_tile} that every tiling by a prime-cardinality cluster is $1$-periodic.
With this and the prime-squared result, no cardinality below the counterexample's $\abs F=8$ can exhibit aperiodic behaviour (given the full affine span hypothesis), except possibly clusters of cardinality $6$.
We discuss this briefly in \S\ref{subsection:threshold open} but leave it open.

The case $\abs F=p^2$ was treated in the author's earlier preprint \cite{khetan_order2}, where the orbit closure was shown to contain a tiling that can be partitioned into at most two $1$-periodic subsets (\cite[Corollary~5.11]{khetan_order2}).
Theorem~\ref{theorem:p squared intro} improves this to a single $1$-periodic point.
The treatment here does not rely on that preprint.
What follows records what is reused and what is new.

Shared with \cite{khetan_order2} are Bhattacharya's dynamical formulation and the spectral framework built on it, together with Lemma~\ref{lemma:rational independence vanishing}, reproduced verbatim from \cite[Lemma~5.6]{khetan_order2}.
Every other relevant lemma from \cite{khetan_order2} is reproved here in a stronger form.
The main one is \cite[Lemma~5.9]{khetan_order2}, which constrains the cluster through a cyclotomic divisibility of the sizes of its sections.
It reappears, strengthened, as Lemma~\ref{lemma:divisible sections}, whose part~(i) records much more about the sections than the divisibility of their sizes, and this finer information is what the present setting requires.
The case analysis in the proof of Theorem~\ref{theorem:p squared one periodic}, which corresponds to \cite[Theorem~5.10]{khetan_order2}, has the same second case, while its first case uses the strengthened constraint.
Carrying that case through, and exploiting the full affine span hypothesis, requires new combinatorial lemmas in \S\ref{subsection:combinatorial preparation} that were not needed in \cite{khetan_order2}.
Further, unlike in the preprint, the argument here makes no use of the Kari--Szabados method.

\subsection{On the use of AI}
The author used Anthropic's Claude Opus~4.8 and OpenAI's GPT-5.5 in this work.
The construction of \S\ref{section:counterexample} is due to GPT-5.5.
The author also used these models to draft the proof of Lemma \ref{lemma:residue block rigidity} and Proposition \ref{proposition:product periodicity} in \S\ref{subsection:combinatorial preparation}, and to improve the exposition throughout.
These tools are not authors of this paper.
Every definition, statement, and proof has been checked by the author, who is solely responsible for the correctness of the results and for any errors that remain.

\subsection{Acknowledgements}

It is a pleasure to thank Prof. Jarkko Kari and Étienne Moutot for valuable e-mail communication and encouragement.
Thanks are due to Gautam Aishwarya for his valuable comments.
The author is immensely grateful to Prof. Siddhartha Bhattacharya, whose remarkable proof of the periodic tiling conjecture in $\Z^2$ made him explore this corner of mathematics.


\section{The Counterexample}
\label{section:counterexample}

This section constructs the cluster $F$ and the $F$-tiling $T$ of Theorem~\ref{theorem:main counterexample} and proves it.
The mechanism is worth previewing, because it explains why the hypotheses of Bhattacharya's Theorem are not violated.
The tiling $T$ is assembled from infinitely many independent one-dimensional ``brick'' patterns running in two transverse directions, with aperiodic phases.
Each brick pattern is individually $1$-periodic, but the two transverse directions never share a period, so $T$ itself has \emph{no} period at all. Because the phases are aperiodic in the strongest sense (Sturmian), no limit of translates of $T$ manages to align into a periodic configuration either.
Bhattacharya's Theorem is untouched: the cluster $F$ still admits biperiodic tilings (the standard brick walls), which do not lie in the orbit closure of this particular $T$.

\subsection{The cluster \texorpdfstring{$F$}{F}}
\label{point:the cluster F}
Let
\begin{equation}
	\label{equation:definition of F}
	\begin{aligned}
		F 
		&= \set{(0,0),(3,0)}\boxplus\set{(0,0),(1,1)}\boxplus\set{(0,0),(0,2)}\\
		&= \lrset{(0,0),(3,0),(0,2),(3,2),(1,1),(4,1),(1,3),(4,3)},
	\end{aligned}
\end{equation}
the Minkowski sum of three two-point sets (every element of $F$ has a unique representation 
$$
	\epsilon_1(3,0)+\epsilon_2(1,1)+\epsilon_3(0,2),
$$
with $\epsilon_i\in\set{0,1}$, so $\abs{F}=8$).
The eight points of $F$ are shown in Figure~\ref{figure:cluster F}.

\medskip\noindent
\emph{$F$ is exact.}
Note that $\Lambda=2\Z\times 4\Z$ has index $8$ in $\Z^2$.
Reducing $F$ modulo $\Lambda$ gives $\set{0,1}\times\set{0,1,2,3}$, which is a set of 8 pairwise distinct representatives of cosets of $\Lambda$.
Hence $F$ itself is a set of 8 pairwise distinct representatives of cosets of $\Lambda$, and we conclude that 
\[
	\Z^2 = F\oplus\Lambda.
\]
So $\Lambda$ is an $F$-tiling, and consequently $F$ is exact.
This tiling is biperiodic.
The content of the section is that $F$ \emph{also} admits the aperiodic tiling $T$ built below.

\medskip\noindent
\emph{$F$ has full affine span.}
From $F-F$ we have the vectors $(3,0)$, $(1,1)$, and $(0,2)$.
Then
$$
	2(1,1)-(0,2)=(2,0)
	\quad \text{ and } \quad
	(3,0)-(2,0)=(1,0).
$$
Together with $(1,1)-(1,0)=(0,1)$, this shows $(1,0),(0,1)\in\langle F-F\rangle$, so $F-F$ generates $\Z^2$ (Definition~\ref{definition:affine span}).

\begin{figure}[H]
	\centering
	\begin{tikzpicture}[scale=0.62]
		\draw[black!10] (-1,-1) grid (5,4);
		\foreach \p in {(0,0),(3,0),(0,2),(3,2),(1,1),(4,1),(1,3),(4,3)} \fill[BurntUmber!82] \p rectangle ++(1,1);
		\draw[BurntUmber!55] (0,0) rectangle (1,1) (3,0) rectangle (4,1) (0,2) rectangle (1,3) (3,2) rectangle (4,3)
			(1,1) rectangle (2,2) (4,1) rectangle (5,2) (1,3) rectangle (2,4) (4,3) rectangle (5,4);
	\end{tikzpicture}
	\caption{
		The cluster $F=\set{(0,0),(3,0)}\boxplus\set{(0,0),(1,1)}\boxplus\set{(0,0),(0,2)}$, drawn as its eight unit cells.
	}
	\label{figure:cluster F}
\end{figure}
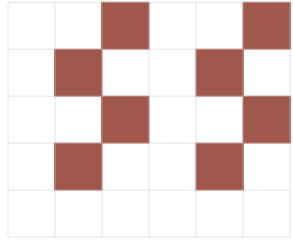

\subsection{Sturmian phases.}
\label{point:sturmian phases}
Fix an irrational $\alpha\in(0,1)$ and a real $\beta$, and define the binary sequence $a=(a_i)_{i\in\Z}$ by
\[
	a_i = \floor{(i+1)\alpha+\beta} - \floor{i\alpha+\beta}\in\set{0,1}.
\]
Such a sequence is called \define{Sturmian}.

\medskip\noindent
Before stating the facts we need, we recall some vocabulary from symbolic dynamics.
The space $\set{0,1}^{\Z}$ of binary sequences carries the product topology, in which it is compact and metrizable, and the \define{shift} $\sigma$ acts on it by $(\sigma a)_i=a_{i+1}$.
A \define{subshift} is a nonempty closed $\sigma$-invariant subset $\Omega\subseteq\set{0,1}^{\Z}$.
If $X$ is a subshift, then for any point $a$ in $X$, we define the \define{orbit} of $a$ as $\mc O_a = \set{\sigma^n x:\ n\in \Z}$ and the \define{orbit closure} of $a$ as $\Omega_a=\operatorname{cl}\set{\sigma^n a:n\in\Z}$.
The orbit closure of $a$ is the ``smallest subshift'' containing it.
A sequence is \define{periodic} if $\sigma^n a=a$ for some nonzero $n$.
Equivalently, $a$ is periodic if its orbit is finite.
A subshift is \define{minimal} if it has no proper nonempty subshift.
Minimality of a subshift $X$ is equivalent to the condition that the orbit of each point in $X$ is dense in $X$.
Finally, a subshift is \define{uniquely ergodic} if it carries exactly one $\sigma$-invariant Borel probability measure.
For such a subshift the \define{frequency} of the symbol $1$, namely 
$$
	\lim_{N\to\infty}
	\frac{1}{2N+1}\sum_{\abs i\le N}a_i,
$$
exists and takes one common value across all of its sequences.
We will use three standard facts about Sturmian sequences \cite[Ch.~2]{lothaire_words}.

\begin{lemma}[Standard facts about Sturmian sequences {\cite[Ch.~2]{lothaire_words}}]
	\label{lemma:sturmian facts}
	Let $a$ be a Sturmian sequence of slope $\alpha$, with $\bar a=(1-a_i)_{i\in \Z}$ its coordinatewise complement and $\Omega_a$ its orbit closure.
	\begin{enumerate}[$(\mathrm{S}1)$]
		\item $a$ is not periodic.
			Moreover $a\neq\bar a$, since the density of $1$'s is $\alpha$ in $a$ but $1-\alpha$ in $\bar a$, and $\alpha\neq 1-\alpha$.
			More generally, the frequency of $1$'s exists and equals $\alpha$ for every sequence in $\Omega_a$ (the Sturmian subshift is uniquely ergodic).
			Since $\alpha\neq 1-\alpha$, the orbit closures $\Omega_a$ and $\Omega_{\bar a}$ are disjoint.
		\item $\Omega_a$ is a minimal subshift.
		\item $\Omega_a$ contains no periodic sequence (in particular neither of the constant sequences).
	\end{enumerate}
\end{lemma}

\noindent Property $(\mathrm{S}3)$ is the one we use: aperiodicity survives every limit of shifts.

\subsection{Brick tilings of the square}

The building block of the construction is the $2\times 2$ square
\[
	Q = \set{(0,0),(1,0),(0,1),(1,1)}.
\]
We need a precise description of two families of $Q$-tilings and their symmetries.
We first set up the families and characterize them combinatorially, with no reference to the phase (Lemma~\ref{lemma:bricks-basic}).
The periodicity properties the construction actually uses, where the phase is Sturmian, are isolated afterwards in Lemma~\ref{lemma:bricks-sturmian}.

\begin{definition}
	Let $a=(a_i)_{i\in\Z}\in\set{0,1}^{\Z}$ be any binary sequence, and put
	\[
		V_a = \set{(2i,\,2j+a_i) : i,j\in\Z},
		\qquad
		H_a = \set{(2i+a_j,\,2j) : i,j\in\Z}.
	\]
	We call $V_a$ the \define{standard vertical brick} and $H_a$ the \define{standard horizontal brick} with phase $a$.
	A \define{vertical brick} (resp.\ \define{horizontal brick}) is any translate of some standard vertical (resp. horizontal) brick.

	Thus a vertical brick is made of $2\times2$ squares stacked into vertical strips two columns wide, all
	aligned to columns of a single parity $e\in\set{0,1}$, which we call its \define{alignment}.
	The standard one
	$V_a$, whose strips are the $\set{2i,2i+1}\times\Z$, has alignment $0$ (\emph{even-aligned}), and a
	$(1,0)$-translate shifts it to alignment $1$ (\emph{odd-aligned}).
\end{definition}

To detect this structure we introduce a family of bookkeeping devices: maps that sample a configuration
along individual rows, and that we will feed into the closedness arguments below.

\begin{definition}
For $e\in\set{0,1}$ and $m\in\Z$ define the \define{slice maps}
\[
	\psi^e_m:\set{0,1}^{\Z^2}\to\set{0,1}^{\Z},
	\qquad
	\psi^e_m(x)=\big(x(2i+e,\,m)\big)_{i\in\Z}.
\]
The map $\psi^e_m$, the \emph{parity-$e$ slice} at row $m$, reads off the bits of $x$ along the row $y=m$ at the columns of parity $e$.
For a $Q$-tiling these bits mark exactly its anchors in that row at those columns.
\end{definition}

Each $\psi^e_m$ is continuous, since every output coordinate is a single input coordinate.

\begin{lemma}[Vertical and horizontal bricks]
	\label{lemma:bricks-basic}
	\begin{enumerate}[$(1)$]
		\item $V_a$ and $H_a$ are $Q$-tilings of $\Z^2$.\footnote{The verification is elementary, and the tiling $V_a$ already appears in \cite[\S1.3]{greenfeld_tao_2020}. We nonetheless give all the details, to keep the section self-contained.}
			Consequently every vertical brick and every horizontal brick is a $Q$-tiling, being a translate of one.
		\item A $Q$-tiling $x$ is a vertical brick if and only if
			\[
				\psi^0_m(x)=\mathbf 0\ \text{for all } m,
				\qquad\text{or}\qquad
				\psi^1_m(x)=\mathbf 0\ \text{for all } m;
			\]
			and the vanishing family then determines $x$ outright:
			\begin{itemize}
				\item if $\psi^1_m(x)=\mathbf 0$ for all $m$, then $x$ is even-aligned, $x=V_a$ with phase $a=\mathbf 1-\psi^0_0(x)$;
				\item if $\psi^0_m(x)=\mathbf 0$ for all $m$, then $x$ is odd-aligned, $x=(1,0)+V_a$ with phase $a=\mathbf 1-\psi^1_0(x)$.
			\end{itemize}
			The analogous statement, with the two coordinates exchanged, characterizes horizontal bricks.
	\end{enumerate}
\end{lemma}
\begin{proof}
	\emph{(1)}
	We show $\set{Q+v : v\in V_a}$ partitions $\Z^2$.
	The squares fall into disjoint vertical strips, and within each strip they tile.
	We give the formal details.
	For $i\in\Z$ let $C_i=\set{2i,2i+1}\times\Z$ be the $i$-th \emph{column-pair
	strip}.
	The strips $C_i$ are disjoint and cover $\Z^2$.
	Every element of $V_a$ has the form $(2i,2j+a_i)$, so the square $Q+(2i,2j+a_i)$ lies in the strip $C_i$.
	Two anchors with distinct $x$-coordinates $2i\neq 2i'$ thus lie in distinct strips $C_i\neq C_{i'}$ and give disjoint squares.
	It therefore suffices to show that, for each fixed $i$, the squares $\set{Q+(2i,2j+a_i):\ j\in\Z}$ partition $C_i$. Now
	$$
		Q+(2i,2j+a_i)
		= \set{(2i, 2j+a_i), (2i+1, 2j+a_i), (2i, 2j+a_i+1), (2i+1, 2j+a_i+1)},
	$$
	fills both columns of $C_i$ across the two rows $\set{2j+a_i,\,2j+a_i+1}$. As $j$ ranges over $\Z$ these
	row-pairs are the consecutive blocks
	$$
		\dots,\set{a_i-2,a_i-1},\set{a_i,a_i+1},\set{a_i+2,a_i+3},\dots,
	$$
	which partition $\Z$. Hence the squares partition $C_i$, and letting $i$ vary gives $Q\oplus V_a=\Z^2$.
	The argument for $H_a$ is identical with the roles of the coordinates exchanged. Since a translate of a
	$Q$-tiling is again one, every vertical brick and every horizontal brick is a $Q$-tiling.

	\smallskip\noindent
	\emph{(2)}
	The forward direction is a direct computation. A vertical brick is, by definition, an arbitrary translate $V_a+(s,t)$ of a standard one.
	Absorbing the even part of $s$ into the column reindexing $i\mapsto i+\lfloor s/2\rfloor$ and the height shift $t$ into the phase (which replaces $a$ by its complement when $t$ is odd) writes it as $(e,0)+V_{a'}$ with $e=s\bmod 2\in\set{0,1}$.
	So it suffices to treat the alignment-$e$ brick $x=(e,0)+V_a$, whose anchors all lie in columns of parity $e$, at heights $\equiv a_i\pmod 2$.
	Hence $\psi^{1-e}_m(x)=\mathbf 0$ for every $m$ (the columns of parity $1-e$ hold no anchors), so one of the two slice families vanishes identically.
	And reading row $0$ gives $\psi^e_0(x)=\bar a$, that is $a=\mathbf 1-\psi^e_0(x)$.

	For the converse, suppose $\psi^1_m(x)=\mathbf 0$ for all $m$.
	Since $\psi^1_m(x)=\big(x(2i+1,m)\big)_{i\in \Z}$, this says precisely that $x$ has no anchor in any odd column, at \emph{any} row.
	Hence every anchor of $x$ lies in an even column, and the strip $\set{2i,2i+1}\times\Z$ is tiled by the squares anchored in column $2i$ alone (a square anchored in an odd column would occupy that odd column together with an even one, and there are no odd-column anchors).

	Fix $i$ and let $H_i=\set{h : (2i,h)\in x}$ be the set of anchor heights in column $2i$.
	The corresponding squares are the translates $Q+(2i,h)$, $h\in H_i$, and $Q+(2i,h)$ occupies the row-pair $\set{h,h+1}$ of the strip $\set{2i,2i+1}\times\Z$.
	Since these squares tile the strip, the pairs $\set{h,h+1}$ ($h\in H_i$) partition $\Z$.

	Such a domino partition of $\Z$ is forced to be a single parity class: if $h\in H_i$ then $h+1$ is the top of this domino, so the integer $h+2$ can only be covered as the bottom of the next domino, giving $h+2\in H_i$, while $h+1\notin H_i$ (its domino would overlap $\set{h,h+1}$).
	Inducting in both directions yields $H_i=h+2\Z$.
	Hence $H_i$ is a single residue class modulo $2$.
	Say $H_i=2\Z+a_i$ with $a_i\in\set{0,1}$.
	The column-$2i$ anchors are therefore just $\set{(2i,\,2j+a_i):j\in\Z}$.
	It follows that
	$$
		x = \set{(2i,\,2j+a_i):i,j\in\Z},
	$$
	which is none other than $V_a$, where $a=(a_i)_{i\in\Z}$.
	Finally, $(2i,0)\in V_a$ if and only if $a_i=0$, so the $i$-th entry of $\psi^0_0(x)$ is $1-a_i=\bar a_i$, that is, $\psi^0_0(x)=\bar a$.
	In other words, $a=\mathbf 1-\psi^0_0(x)$.
	
	If instead $\psi^0_m(x)=\mathbf 0$ for all $m$, the same argument with the two column parities exchanged
	shows $x$ has all its anchors in odd columns, so $x=(1,0)+V_a$ is the odd-aligned brick. The same computation as above, now reading the odd columns, gives the phase $a=\mathbf 1-\psi^1_0(x)$.
	The characterization of horizontal bricks is in turn the same argument with the two coordinates exchanged.
\end{proof}

We now specialize the phase to a Sturmian sequence and record the periodicity facts the construction relies
on.

\begin{lemma}[Periodicity of Sturmian bricks]
	\label{lemma:bricks-sturmian}
	Let $a$ be a Sturmian.
	\begin{enumerate}[$(1)$]
		\item The period group (Definition~\ref{definition:periodicity}) of $V_a$ is $\set{0}\times 2\Z$, and that of $H_a$ is $2\Z\times\set{0}$.\footnote{Again elementary, and consistent with the observation in \cite[\S1.3]{greenfeld_tao_2020} that $V_a$ is $\langle(0,2)\rangle$-periodic and biperiodic only when $a$ is periodic. We give the full computation regardless.}
		\item Every tiling in the orbit closure $\overline{\Z^2\cdot V_a}$ is a vertical brick whose phase lies in $\Omega_a\cup\Omega_{\bar a}$, and hence has period group $\set{0}\times 2\Z$.
			In particular it is invariant under $(0,2)$ and admits no horizontal period.
			The symmetric statement holds for $H_a$.
	\end{enumerate}
\end{lemma}

\begin{proof}
	\emph{(1)}
	Clearly $V_a$ is invariant under $\set{0}\times 2\Z$. For the converse, suppose $(m,n)$ is a period of
	$V_a$.
	We show $m=0$ and $n$ even.
	Since $(0,a_0)\in V_a$, we must also have 
	$$
		(m,n)+(0,a_0)=(m,n+a_0)\in V_a,
	$$
	and as every element of $V_a$ has even $x$-coordinate, $m$ is even.
	Say $m=2k$.
	For each $(i,j)\in\Z^2$ we have $(2i,2j+a_i)\in V_a$.
	Hence
	\[
		(m,n)+(2i,2j+a_i)=(2(k+i),\,n+2j+a_i)\in V_a,
	\]
	which forces $n+2j+a_i=2t+a_{k+i}$ for some $t\in\Z$.
	This implies that $n+a_i\equiv a_{k+i}\pmod 2$ for all $i$.
	If $n$ is even, then $a_{i+k}\equiv a_i\pmod 2$, so $a_{i+k}=a_i$ for all $i$, that is $\sigma^k a=a$.
	By $(\mathrm{S}1)$ of Lemma~\ref{lemma:sturmian facts}, no nonzero shift fixes a Sturmian sequence, so $k=0$ and $m=0$. If $n$ is odd, then
	$a_{i+k}=\bar a_i$ for all $i$, that is $\sigma^k a=\bar a$, impossible by $(\mathrm{S}1)$ since a shift
	preserves the density of $1$'s ($\alpha$ for $a$, but $1-\alpha\neq\alpha$ for $\bar a$). Hence $m=0$ and
	$n$ is even. The computation for $H_a$ is the transpose.

	\smallskip\noindent
	\emph{(2)}
	We observe that
	\[
		(2,0)\cdot V_a = V_{\sigma^{-1}a},
		\qquad
		(1,0)\cdot V_a = \set{(2i+1,2j+a_i):\ (i,j)\in\Z^2},
		\qquad
		(0,1)\cdot V_a = V_{\bar a},
	\]
	the middle set being the odd-aligned brick $(1,0)+V_a$ of phase $a$.
	Hence every translate of $V_a$ is a vertical brick whose phase lies in 
	$$
		\set{\sigma^n a:n\in\Z}\cup\set{\sigma^n\bar a:n\in\Z}\subseteq\Omega_a\cup\Omega_{\bar a}.
	$$
	Let $\mc B$ be the set of vertical bricks with phase in $\Omega_a\cup\Omega_{\bar a}$.
	Thus, by the above discussion, we know that $\Z^2\cdot V_a\subseteq\mc B$.
	We claim $\mc B$ is closed, whence it follows that $\overline{\Z^2\cdot V_a}\subseteq\mc B$.

	Split $\mc B$ according to alignment: $\mc B=\mc B_0\cup\mc B_1$, where $\mc B_e$ is the set of alignment-$e$ vertical bricks whose phase lies in $\Omega_a\cup\Omega_{\bar a}$.
	It suffices to show each $\mc B_e$ is closed, and we do so by exhibiting $\mc B_e$ as the intersection of two closed sets, one for each of its two defining conditions: the \emph{brick condition} (being an alignment-$e$ vertical brick) and the \emph{phase condition} (having phase in $\Omega_a\cup\Omega_{\bar a}$).

	For the brick condition, Lemma~\ref{lemma:bricks-basic}(2) says a $Q$-tiling $x$ is an alignment-$e$ vertical brick if and only if $\psi^{1-e}_m(x)=\mathbf 0$ for all $m$, so these bricks form
	\[
		B_e := \set{Q\text{-tilings}}\ \cap\ \bigcap_{m\in\Z}\big(\psi^{1-e}_m\big)^{-1}(\mathbf 0),
	\]
	which is closed because the $\psi^{1-e}_m$ are continuous, $\set{\mathbf 0}$ is closed, and the $Q$-tilings form a closed subset of $\set{0,1}^{\Z^2}$.

	For the phase condition, recall from Lemma~\ref{lemma:bricks-basic}(2) that on an alignment-$e$ vertical brick the phase is $\Phi_e(x):=\mathbf 1-\psi^e_0(x)$, and the map $\Phi_e\colon\set{0,1}^{\Z^2}\to\set{0,1}^{\Z}$ is continuous.
	Hence an $x\in B_e$ has phase in $\Omega_a\cup\Omega_{\bar a}$ precisely when $x\in\Phi_e^{-1}(\Omega_a\cup\Omega_{\bar a})$, and this preimage is closed, being the preimage of the closed set $\Omega_a\cup\Omega_{\bar a}$ under the continuous map $\Phi_e$.
	Therefore $\mc B_e=B_e\cap\Phi_e^{-1}(\Omega_a\cup\Omega_{\bar a})$ is closed, and so is $\mc B=\mc B_0\cup\mc B_1$.

	It remains to show that every $U\in\overline{\Z^2\cdot V_a}$ has period group $\set{0}\times 2\Z$. Its phase $\vp$ lies in
	$\Omega_a\cup\Omega_{\bar a}$, on which the density argument of $(\mathrm{S}1)$ of Lemma~\ref{lemma:sturmian facts} is uniform (density of
	$1$'s equal to $\alpha$ on $\Omega_a$, to $1-\alpha$ on $\Omega_{\bar a}$) and, by $(\mathrm{S}3)$, no phase is periodic.
	So the computation of part~$(1)$ applies verbatim to the standard brick $V_{\vp}$,
	giving it period group $\set{0}\times 2\Z$. Since period groups are translation-invariant and $U$ is a
	translate of $V_{\vp}$, the period group of $U$ is exactly $\set{0}\times 2\Z$: invariant under $(0,2)$,
	with no horizontal period. The statement for $H_a$ is the transpose.
\end{proof}

\noindent
Figure~\ref{figure:bricks} shows the two brick families.
The defining feature is that the fault lines run in a single direction: a vertical brick has unbroken \emph{vertical} fault lines and staggered horizontal joins, so it can be shifted vertically but, with an aperiodic phase, never horizontally.

\begin{figure}[H]
	\centering
	\subcaptionbox{Vertical brick $V_a$: period $(0,2)$, no horizontal period.\label{figure:vertical brick}}[0.47\linewidth]{%
	\begin{tikzpicture}[scale=0.42]
		\clip (0,0) rectangle (10,8);
		\foreach \k/\o in {0/0,1/1,2/0,3/0,4/1}{
			\foreach \m in {-1,0,1,2,3,4}{
				\fill[BurntUmber!18] (2*\k,2*\m+\o) rectangle (2*\k+2,2*\m+\o+2);
				\draw[BurntUmber!75,thin] (2*\k,2*\m+\o) rectangle (2*\k+2,2*\m+\o+2);
			}
		}
		\foreach \k in {0,1,2,3,4,5} \draw[darkbrown,very thick] (2*\k,0) -- (2*\k,8);
	\end{tikzpicture}}%
	\hfill
	\subcaptionbox{Horizontal brick $H_a$: period $(2,0)$, no vertical period.\label{figure:horizontal brick}}[0.47\linewidth]{%
	\begin{tikzpicture}[scale=0.42]
		\clip (0,0) rectangle (10,8);
		\foreach \k/\o in {0/0,1/1,2/1,3/0}{
			\foreach \m in {-1,0,1,2,3,4,5}{
				\fill[gold!22] (2*\m+\o,2*\k) rectangle (2*\m+\o+2,2*\k+2);
				\draw[darkbrown!70,thin] (2*\m+\o,2*\k) rectangle (2*\m+\o+2,2*\k+2);
			}
		}
		\foreach \k in {0,1,2,3,4} \draw[darkbrown,very thick] (0,2*\k) -- (10,2*\k);
	\end{tikzpicture}}
	\caption{The two brick tilings of the $2\times2$ square $Q$. Thick lines are the unbroken fault
	lines. The staggered offsets across strips are governed by the (aperiodic) Sturmian phase, which is why
	there is no period in the transverse direction.}
	\label{figure:bricks}
\end{figure}

\subsection{Assembling \texorpdfstring{$F$}{F}-tilings from three squares}

The link between $F$ and the square $Q$ is an explicit dictionary.
It says that to build an $F$-tiling it suffices to choose \emph{three independent} $Q$-tilings, indexed by the residue of the first coordinate modulo $3$, and interleave them.

\begin{proposition}[The interleaving dictionary]
	\label{proposition:dictionary}
	Given three subsets $\tau_0,\tau_1,\tau_2\subseteq\Z^2$, define
	\begin{equation}
		\label{equation:Psi}
		\Psi(\tau_0,\tau_1,\tau_2) = \lrset{(3i+d,\,2j) : d\in\set{0,1,2},\ (i,j)\in\tau_d}\subseteq\Z^2.
	\end{equation}
	Then $\Psi(\tau_0,\tau_1,\tau_2)$ is an $F$-tiling of $\Z^2$ if and only if each $\tau_d$ is a
	$Q$-tiling of $\Z^2$.
\end{proposition}
\begin{proof}
	Recall from \eqref{equation:definition of F} that every $f\in F$ can be uniquely written as
	$$
		f=(3\epsilon_1+\epsilon_2,\ \epsilon_2+2\epsilon_3)
	$$
	with $\epsilon_1,\epsilon_2,\epsilon_3\in\set{0,1}$.
	Write $T=\Psi(\tau_0,\tau_1,\tau_2)$.
	We must decide, for each $p=(x,y)\in\Z^2$, how many pairs $(f,t)\in F\times T$ satisfy $f+t=p$.
	The set $T$ is an $F$-tiling precisely when this count is always $1$.

	Take $t=(3i+d,2j)$ with $(i,j)\in\tau_d$. The equation $f+t=p$ reads
	\[
		x = 3\epsilon_1+\epsilon_2+3i+d,
		\qquad
		y = \epsilon_2+2\epsilon_3+2j.
	\]
	We solve it from the outside in. The second coordinate forces the parity bit
	\[
		\epsilon_2 = y\bmod 2,
		\qquad\text{and then}\qquad
		\epsilon_3+j = \tfrac{1}{2}(y-\epsilon_2)=:y',
		\quad\text{so } j=y'-\epsilon_3 .
	\]
	With $\epsilon_2$ now fixed, the first coordinate forces the residue $d$ and an integer $x'$:
	\[
		d = (x-\epsilon_2)\bmod 3,
		\qquad
		\epsilon_1+i = \tfrac{1}{3}(x-\epsilon_2-d)=:x',
		\quad\text{so } i=x'-\epsilon_1 .
	\]
	Thus $\epsilon_2$, $d$, $x'$, $y'$ are \emph{determined} by $p$, and the only remaining freedom is the choice of $\epsilon_1,\epsilon_3\in\set{0,1}$. For each such choice we obtain the candidate lattice point
	\[
		(i,j)=(x'-\epsilon_1,\,y'-\epsilon_3),
	\]
	and the pair $(f,t)$ is admissible exactly when this point lies in $\tau_d$. As $(\epsilon_1,\epsilon_3)$ ranges over $\set{0,1}^2$, the four candidate points are
	\[
		\lrset{(x'-\epsilon_1,\,y'-\epsilon_3) : \epsilon_1,\epsilon_3\in\set{0,1}}
		= (x',y') - Q .
	\]
	Therefore the number of admissible pairs $(f,t)$ with $f+t=p$ equals the number of points of $(x',y')-Q$ that lie in $\tau_d$, namely 
	$$
		\sum_{q\in Q}1_{\tau_d}\big((x',y')-q\big).
	$$
	Consequently $T$ is an $F$-tiling if and only if, for every $p$ (equivalently, for every residue
	$d$ and every $(x',y')\in\Z^2$),
	\[
		\sum_{q\in Q}1_{\tau_d}\big((x',y')-q\big)=1 .
	\]
	By the tiling criterion \eqref{equation:tiling condition}, this is exactly the statement that
	each $\tau_d$ is a $Q$-tiling. As $d=(x-\epsilon_2)\bmod 3$ takes all three values and $(x',y')$
	ranges over all of $\Z^2$ as $p$ does, the three conditions (one per $d$) are independent, completing the
	proof.
\end{proof}

\noindent
The dictionary loses no information: the three ingredients can be read back off the assembled tiling.

\begin{lemma}[Injectivity of $\Psi$]
	\label{lemma:Psi injective}
	For all subsets $\tau_0,\tau_1,\tau_2\subseteq\Z^2$ and every $d\in\set{0,1,2}$,
	\begin{equation}
		\label{equation:Psi recovery}
		\tau_d=\set{(i,j)\in\Z^2 : (3i+d,\,2j)\in\Psi(\tau_0,\tau_1,\tau_2)}.
	\end{equation}
	In particular $\Psi$ is injective, and its image is exactly the family of subsets of $\Z^2$ whose points all
	have even second coordinate.
\end{lemma}
\begin{proof}
	By \eqref{equation:Psi} a point $(x, y)$ of $\Psi(\tau_0,\tau_1,\tau_2)$ has the form $(3i+d,2j)$ with $d\in\set{0,1,2}$ and $(i,j)\in\tau_d$, and its coordinates determine $d$ as the residue of the first coordinate modulo $3$ and then $(i,j)=\big(\tfrac{x-d}{3},\tfrac{y}{2}\big)$.
	Thus $(3i+d,2j)$ lies in $\Psi(\tau_0,\tau_1,\tau_2)$ if and only if $(i,j)\in\tau_d$, which is \eqref{equation:Psi recovery}.
	Since it expresses each $\tau_d$ in terms of the set $\Psi(\tau_0,\tau_1,\tau_2)$, the map $\Psi$ is injective.
	For the image, note first that every point of $\Psi(\tau_0,\tau_1,\tau_2)$ has even second coordinate $2j$, so the image is contained in the family of subsets of $\Z\times2\Z$.
	Conversely, given any $S\subseteq\Z\times2\Z$, define 
	$$
		\tau_d=\set{(i,j):(3i+d,2j)\in S}
	$$
	for $d\in\set{0,1,2}$.
	Every point of $S$ has even second coordinate $2j$ and first coordinate uniquely of the form $3i+d$ with $d=x\bmod 3$, so it lies in $\Psi(\tau_0,\tau_1,\tau_2)$.
	Further, every point of $\Psi(\tau_0,\tau_1,\tau_2)$ lies in $S$ by the definition of the $\tau_d$.
	Hence $\Psi(\tau_0,\tau_1,\tau_2)=S$, so every subset of $\Z\times2\Z$ is in the image.
\end{proof}

\begin{lemma}[Equivariance of $\Psi$]
	\label{lemma:equivariance}
	For all $Q$-tilings $\tau_0,\tau_1,\tau_2$,
	\begin{align}
		(0,2)\cdot\Psi(\tau_0,\tau_1,\tau_2) &= \Psi\big(\tau_0+(0,1),\,\tau_1+(0,1),\,\tau_2+(0,1)\big),
		\label{equation:equivariance vertical}\\
		(3,0)\cdot\Psi(\tau_0,\tau_1,\tau_2) &= \Psi\big(\tau_0+(1,0),\,\tau_1+(1,0),\,\tau_2+(1,0)\big),
		\label{equation:equivariance horizontal}\\
		(1,0)\cdot\Psi(\tau_0,\tau_1,\tau_2) &= \Psi\big(\tau_2+(1,0),\,\tau_0,\,\tau_1\big).
		\label{equation:equivariance permutation}
	\end{align}
	Moreover $(0,1)\cdot\Psi(\tau_0,\tau_1,\tau_2)$ has all anchors with \emph{odd} second coordinate, so it is not of the form $\Psi(\,\cdot\,)$.
	The image of $\Psi$ is invariant under the index-$2$ subgroup $\Z\times 2\Z$, and is interchanged with its complementary parity class by $(0,1)$.
\end{lemma}
\begin{proof}
	Each identity follows by applying the shift to a generic anchor $(3i+d,2j)$ of $\Psi(\tau_0,\tau_1,\tau_2)$
	and rewriting the result in the normal form $(3i'+d',2j')$ with $d'\in\set{0,1,2}$.

	\smallskip\noindent
	\emph{The vertical shift \eqref{equation:equivariance vertical}.} Here $(0,2)$ sends
	$(3i+d,2j)\mapsto(3i+d,2(j+1))$, which keeps the residue $d$ and replaces $j$ by $j+1$, i.e.\ shifts every
	$\tau_d$ by $(0,1)$.

	\smallskip\noindent
	\emph{The horizontal shift \eqref{equation:equivariance horizontal}.} Here $(3,0)$ sends
	$(3i+d,2j)\mapsto(3(i+1)+d,2j)$, shifting every $\tau_d$ by $(1,0)$.

	\smallskip\noindent
	\emph{The unit horizontal shift \eqref{equation:equivariance permutation}.} Here $(1,0)$ sends $(3i+d,2j)\mapsto(3i+d+1,2j)$. For $d\in\set{0,1}$ this is $(3i+(d+1),2j)$, lowering the index $d+1$ onto the
	tiling $\tau_d$, while for $d=2$ it is $(3(i+1)+0,2j)$, placing $\tau_2$ shifted by $(1,0)$ into residue
	$0$. Reading off the residue classes gives the stated cyclic permutation.

	Finally, the parity claim is immediate since $\Psi$ produces only even second coordinates.
\end{proof}

\subsection{The aperiodic tiling \texorpdfstring{$T$}{T} and the proof}

Fix any three Sturmian sequences $a,b,c$. Their slopes and offsets play no role in what follows (only that each sequence is Sturmian, hence aperiodic), so for definiteness take all three of slope $\alpha=\tfrac{\sqrt5-1}{2}$ (one could even use a single sequence in all three slots).
Put two vertical bricks on the even residue classes and one horizontal brick on the middle class:
\begin{equation}
	\label{equation:definition of T}
	T = \Psi\big(V_a,\ H_b,\ V_c\big).
\end{equation}
By Lemma~\ref{lemma:bricks-basic}(1) each ingredient brick is a $Q$-tiling, so by Proposition~\ref{proposition:dictionary} $T$ is an $F$-tiling.
Also, it is not $1$-periodic.
Suppose $v=(m,n)$ is a period of $T$, and write $n=2n'+\delta$ with $\delta\in\set{0,1}$ and $m=3m'+s$ with $s\in\set{0,1,2}$, so that $v=m'(3,0)+n'(0,2)+s(1,0)+\delta(0,1)$.
Lemma~\ref{lemma:equivariance} tells us how each summand acts on the three ingredient tilings, and we use this to force $v=0$ in three steps.
\begin{enumerate}[i)]
	\item \emph{We show that $n$ is even, that is, $\delta=0$.}
		Every anchor of $T$ has even second coordinate, since $T$ lies in the image of $\Psi$, which consists of the subsets of $\Z^2$ with even second coordinate (Lemma~\ref{lemma:Psi injective}).
		The anchors of $v\cdot T$ then have second coordinate $2j+n$. Were $n$ odd, these would all be odd and hence disjoint from the anchors of $T$, so $v\cdot T\neq T$.
	\item \emph{We show that $s=0$.}
		Since $n$ is even, the anchors of $v\cdot T$ are again even, so $v\cdot T$ lies in the image of $\Psi$ (Lemma~\ref{lemma:Psi injective}).
		Write $v\cdot T=\Psi(\sigma_0,\sigma_1,\sigma_2)$.
		By the equivariance identities \eqref{equation:equivariance vertical}--\eqref{equation:equivariance permutation}, each $\sigma_d$ is a translate of one of $V_a,H_b,V_c$, and the cyclic shift $d\mapsto d+s$ carries the orientation pattern $(\text{vertical},\text{horizontal},\text{vertical})$ of $T$ to that of $(\sigma_0,\sigma_1,\sigma_2)$. 
		Using $v\cdot T=T$, the injectivity of $\Psi$ (Lemma~\ref{lemma:Psi injective}) forces $\sigma_d$ to equal the $d$-th ingredient brick of $T$ for every $d$, so the two orientation patterns coincide.
		But $(\text{vertical},\text{horizontal},\text{vertical})$ has its only horizontal entry in residue $1$, and a cyclic shift fixes it only when $s=0$.
		(A translate of a vertical brick is never the horizontal brick $H_b$: their period groups, Lemma~\ref{lemma:bricks-sturmian}(1), are the transverse lines $\set{0}\times2\Z$ and $2\Z\times\set{0}$.)
		Hence $s=0$.
	\item \emph{We show that $v=0$.} With $\delta=s=0$ the shift $v=m'(3,0)+n'(0,2)$ permutes nothing and moves every ingredient brick by the same vector $(m',n')$, by \eqref{equation:equivariance vertical}--\eqref{equation:equivariance horizontal}.
		So $v\cdot T=T$ forces $(m',n')$ to be a common period of $V_a$ and $H_b$.
		By Lemma~\ref{lemma:bricks-sturmian}(1) these period groups are $\set{0}\times2\Z$ and $2\Z\times\set{0}$, which meet only in $0$. Therefore $(m',n')=0$ and $v=0$.
\end{enumerate}
\noindent
The point of Theorem~\ref{theorem:main counterexample} is the stronger assertion that not even a \emph{limit} of translates of $T$ is periodic.
The proof rests on a single structural observation: the brick \emph{orientations} survive passage to the
orbit closure.

\begin{lemma}[Orientation is preserved in the orbit closure]
	\label{lemma:orientation preserved}
	Recall $T=\Psi(V_a,H_b,V_c)$ from \eqref{equation:definition of T}, the brick wall with two vertical Sturmian layers (residues $0,2$ modulo 3) and one horizontal Sturmian layer (residue $1$ modulo 3).
	Every tiling $U\in\overline{\Z^2\cdot T}$ can be written, after possibly translating by $(0,1)$, as $U=\Psi(\sigma_0,\sigma_1,\sigma_2)$ where exactly two of $\sigma_0,\sigma_1,\sigma_2$ are vertical Sturmian bricks and the third is a horizontal Sturmian brick.
	In particular each $\sigma_d$ has the period group given by Lemma~\ref{lemma:bricks-sturmian}(2).
\end{lemma}
\begin{proof}
	Let $U=\lim_{n} w_n\cdot T$ for some sequence $(w_n)_{n\geq1}$ in $\Z^2$.
	Say $w_n=(p_n,q_n)$ for each $n$.
	Passing to a subsequence, we may assume $p_n\bmod 3$ is constant, equal to some $r\in\set{0,1,2}$, and $q_n\bmod 2$ is constant.
	If that parity is odd, replace each $w_n$ by $w_n+(0,1)$.
	This only translates the limit, replacing $U$ by $(0,1)\cdot U$, so we may assume every $q_n$ is even.
	Then every anchor of every $w_n\cdot T$ has even second coordinate (Lemma~\ref{lemma:equivariance}), and hence so does $U$.
	Thus $U$ lies in the image of $\Psi$, so by Lemma~\ref{lemma:Psi injective} it equals $\Psi(\sigma_0,\sigma_1,\sigma_2)$ for a unique triple $(\sigma_0,\sigma_1,\sigma_2)$.
	Since $U$ is an $F$-tiling, each $\sigma_d$ is a $Q$-tiling by Proposition~\ref{proposition:dictionary}.

	It remains to identify each $\sigma_d$. We first compute the residue classes of the approximants $w_n\cdot T$, and then pass to the limit.
	Fix $n$.
	Since $q_n$ is even, $w_n\cdot T$ has only even-height anchors and so lies in the image of $\Psi$ (Lemma~\ref{lemma:equivariance}).
	Write
	\[
		w_n\cdot T=\Psi\big(\sigma^{(n)}_0,\sigma^{(n)}_1,\sigma^{(n)}_2\big).
	\]
	What we must determine is \emph{which} ingredient brick of $T$ each $\sigma^{(n)}_d$ is a translate of. For this, decompose
	\[
		w_n=k(3,0)+\ell(0,2)+r(1,0),
		\qquad p_n=3k+r,\ \ q_n=2\ell
	\]
	(legitimate since $q_n$ is even and $r=p_n\bmod3$).
	Here $k,\ell\in\Z$ depend on the fixed $n$, whereas $r$ is the residue fixed once and for all on the subsequence. We read off each summand from the equivariance lemma.
	By \eqref{equation:equivariance vertical}--\eqref{equation:equivariance horizontal}, the summands $k(3,0)$ and $\ell(0,2)$ translate all three ingredient bricks by $(k,\ell)$ and leave their residue labels in place.
	The remaining summand is the shift $(r,0)$, that is, the unit shift $(1,0)$ applied $r$ times.
	By \eqref{equation:equivariance permutation} a single $(1,0)$ moves the ingredient brick sitting in residue class $d-1$ into residue class $d$ (translating the class that wraps around from $2$ to $0$ by $(1,0)$).
	Applying $(1,0)$ a total of $r$ times therefore lands the ingredient that started in residue $d-r$ in residue $d$ (indices modulo $3$).
	Since the ingredients of $T$ are $V_a,H_b,V_c$ in residues $0,1,2$, we conclude that $\sigma^{(n)}_d$ is a translate of the residue-$(d-r)$ ingredient of $T$.
	Writing $B_0=V_a$, $B_1=H_b$, $B_2=V_c$, this says $\sigma^{(n)}_d=B_{d-r}+t_{d,n}$ for some $t_{d,n}\in\Z^2$. In particular $\sigma^{(n)}_d\in\Z^2\cdot B_{d-r}$.

	Now pass to the limit.
	For tilings in the image of $\Psi$, the recovery formula \eqref{equation:Psi recovery} exhibits the $d$-th residue class as a \emph{continuous} function of the tiling, since each of its coordinates is a single coordinate of the underlying set.
	Both $U$ and every $w_n\cdot T$ lie in the image of $\Psi$, and $w_n\cdot T\to U$ as $n\to \infty$.
	Applying this continuous map along the sequence gives $\sigma_d=\lim_n\sigma^{(n)}_d$ for each $d$.
	Thus $\sigma_d=\lim_n\sigma^{(n)}_d$ is a limit of elements of $\Z^2\cdot B_{d-r}$, hence lies in $\overline{\Z^2\cdot B_{d-r}}$. By Lemma~\ref{lemma:bricks-sturmian}(2) it is a vertical Sturmian brick when $B_{d-r}$ is $V_a$ or $V_c$, and a horizontal Sturmian brick when it is $H_b$.
	Since $d\mapsto d-r$ is a bijection of $\set{0,1,2}$, the multiset of orientations of $(\sigma_0,\sigma_1,\sigma_2)$ is the same as that of $(V_a,H_b,V_c)$: two vertical and one horizontal.
\end{proof}

\begin{proof}[Proof of Theorem~\ref{theorem:main counterexample}]
	The cluster $F$ of \eqref{equation:definition of F} is exact with full affine span
	(Section~\ref{point:the cluster F}), and $T$ of \eqref{equation:definition of T} is an $F$-tiling. It
	remains to show $\overline{\Z^2\cdot T}$ contains no $1$-periodic $F$-tiling.

	Suppose, for contradiction, that some $U\in\overline{\Z^2\cdot T}$ has a nonzero period $w=(p,q)$.
	A period is unchanged by translating $U$, so by Lemma~\ref{lemma:orientation preserved} we may take $U=\Psi(\sigma_0,\sigma_1,\sigma_2)$ with orientation multiset $\set{V,V,H}$.
	Without loss of generality, we may assume that the horizontal brick sits in residue class $d_0$, and the other two classes are vertical.

	Note that $q$ is even.
	Indeed $w\cdot U=U$ and $U$ has only even-height anchors. If $q$ were odd then $w\cdot U$ would have only odd-height anchors (Lemma~\ref{lemma:equivariance}, parity clause), contradicting $w\cdot U=U$. Write $q=2q'$.
	We make cases based on the residue of $p$ modulo $3$.
	\begin{enumerate}[(1)]
		\item \emph{$p\not\equiv 0\pmod 3$.} By \eqref{equation:equivariance permutation} the shift by $(1,0)$
		      cyclically permutes the three residue classes, so $w=(p,2q')$ permutes them by $p\bmod 3$, a
		      nontrivial cyclic permutation. Applying $w$ to $U$ therefore moves the orientation of class $d_0$
		      to class $d_0+p\not\equiv d_0$. But $U$ has a \emph{unique} horizontal class, so $w\cdot U$ has its
		      horizontal class in a different position than $U$ does. In particular $w\cdot U\neq U$. This
		      contradicts $w\cdot U=U$.

		\item \emph{$p\equiv 0\pmod 3$.} Write $p=3p'$. By
		      \eqref{equation:equivariance vertical}--\eqref{equation:equivariance horizontal} the shift
		      $w=(3p',2q')$ fixes each residue class and acts within it as the square-tiling shift by
		      $(p',q')$. Thus $w\cdot U=U$ forces $(p',q')$ to be a period of every $\sigma_d$. But two of the
		      $\sigma_d$ are vertical bricks, with period group $\set{0}\times 2\Z$, and one is a horizontal
		      brick, with period group $2\Z\times\set{0}$ (Lemma~\ref{lemma:bricks-sturmian}(2)). Hence
		      \[
		      	(p',q')\in\big(\set{0}\times 2\Z\big)\cap\big(2\Z\times\set{0}\big)=\set{(0,0)},
		      \]
		      so $p'=q'=0$ and therefore $w=(3p',2q')=(0,0)$, contradicting $w\neq 0$.
	\end{enumerate}
	Both cases are impossible, so $\overline{\Z^2\cdot T}$ contains no $1$-periodic $F$-tiling. This proves
	Theorem~\ref{theorem:main counterexample}, and Conjecture~\ref{conjecture:orbit closure one periodicity}
	is false.
\end{proof}

\begin{remark}
	The vertical brick $V_a$ is exactly the tiling of \cite[\S1.3, eq.~(1)]{greenfeld_tao_2020} by the square $\set{0,1}^2$ with arbitrary phase $a$.
	It is always periodic along $(0,2)$, but biperiodic only when $a$ is periodic.
	An orbit closure with no periodic point at all is the kind of object studied in \cite{grandjean_menibus_vanier_2018}.
\end{remark}


\section{Clusters of Prime-Squared Cardinality}
\label{section:prime squared}

The counterexample of Section~\ref{section:counterexample} is a cluster of eight cells whose orbit closure, for a suitable tiling, contains no $1$-periodic $F$-tiling.
How small can a cluster be before this happens?
It was shown in \cite{szegedy_algorithms_to_tile} that if a cluster has \emph{prime} cardinality then every tiling is $1$-periodic.
Thus no cluster with $\abs F$ prime can exhibit the phenomenon.
The next case is $\abs F=p^2$, and this section disposes of it by showing that for every exact cluster of full affine span with $\abs F=p^2$, the orbit closure of any $F$-tiling contains a $1$-periodic $F$-tiling (Theorem~\ref{theorem:p squared one periodic}), so Conjecture~\ref{conjecture:orbit closure one periodicity} holds in this case.

Together with the prime case, this rules out every cardinality below $8$ except $\abs F=6$: the primes $2,3,5,7$ by Szegedy, and $4=2^2$ by the present theorem.
The smallest cardinality at which the orbit closure can fail to contain a $1$-periodic $F$-tiling is therefore either $6$ or $8$.
The remaining case $\abs F=6$ is discussed in \S\ref{subsection:threshold open} below, where we leave it open. The eight-cell counterexample, in which $8=2^3$, also shows that
Theorem~\ref{theorem:p squared one periodic} cannot be pushed from $p^2$ up to $p^3$.
The proof is analytic, following Bhattacharya's ergodic-theoretic proof of the
periodic tiling conjecture \cite{bhattacharya_tilings}.\footnote{References to \cite{bhattacharya_tilings} are to the arXiv version.}

\subsection{Arithmetic preparations}

We need some preparation.
Call $\gamma\in \mathbb S^1$ \define{irrational} if it is not a
root of unity. Call $\gamma_1,\dots,\gamma_m\in \mathbb S^1$ \define{rationally independent} if no nontrivial
character of $(\mathbb S^1)^m$ kills $(\gamma_1,\dots,\gamma_m)$.

\smallskip\noindent
The following lemma appears in \cite[Lemma 3.5]{khetan_prime_squared} and \cite[Lemma~5.6]{khetan_order2}.
We include a proof for completeness.

\begin{lemma}[A vanishing criterion via rational independence]
	\label{lemma:rational independence vanishing}
	Let $\gamma_1,\dots,\gamma_n\in \mathbb S^1$ be irrational and $x_1,\dots,x_n\in\C$. If
	\[
		(\gamma_1^k-1)x_1+\cdots+(\gamma_n^k-1)x_n\in\Z\qquad\text{for all integers } k\ge 0,
	\]
	then this expression equals $0$ for every $k$.
\end{lemma}
\begin{proof}
	Write $S_k=\sum_{i=1}^{n}(\gamma_i^k-1)x_i$.
	We must show $S_k=0$ for every $k\ge 0$.
	Let $G=\ab{\gamma_1,\dots,\gamma_n}$ be the subgroup of $\mathbb S^1$ they generate.
	It is finitely generated and abelian, and its torsion subgroup $G_{\mathrm{tor}}$ is finite.
	Being a finite subgroup of $\mathbb S^1$ it is cyclic, so $G_{\mathrm{tor}}=\ab{\zeta}$ for a primitive $q$-th root of unity $\zeta$ (with $q=1$ if $G_{\mathrm{tor}}$ is trivial).
	The quotient $G/G_{\mathrm{tor}}$ is finitely generated and torsion-free, hence free.
	Choosing $\beta_1,\dots,\beta_d\in G$ whose images form a basis of $G/G_{\mathrm{tor}}$, we can write every element of $G$, in particular each $\gamma_i$, uniquely as
	\[
		\gamma_i=\zeta^{c_{i0}}\beta_1^{c_{i1}}\cdots\beta_d^{c_{id}},\qquad c_{ij}\in\Z .
	\]
	The $\beta_1,\dots,\beta_d$ are rationally independent: a relation $\beta_1^{a_1}\cdots\beta_d^{a_d}=1$ would, since they form a basis of $G/G_{\mathrm{tor}}$, force $a_1=\dots=a_d=0$.
	Set $c_i=(c_{i1},\dots,c_{id})\in\Z^d$.
	Each $\gamma_i$ is irrational, hence not a root of unity, so $c_i\neq 0$ for every $i$ (otherwise $\gamma_i=\zeta^{c_{i0}}\in G_{\mathrm{tor}}$).

	Fix a residue $\rho\in\set{0,1,\dots,q-1}$ and define a Laurent polynomial on $(\mathbb S^1)^d$ by
	\[
		f_\rho(U_1,\dots,U_d)
		=\sum_{i=1}^{n}\lrp{\zeta^{\rho c_{i0}}\beta_1^{\rho c_{i1}}\cdots\beta_d^{\rho c_{id}}\,
		U_1^{c_{i1}}\cdots U_d^{c_{id}}-1}x_i .
	\]
	Put $\delta_j=\beta_j^{\,q}$. The $\delta_j$ are again rationally independent, since
	$\delta_1^{a_1}\cdots\delta_d^{a_d}=1$ gives $\beta_1^{qa_1}\cdots\beta_d^{qa_d}=1$, whence
	$qa_j=0$ and so $a_j=0$ for all $j$. For $k=\rho+qt$ with $t\ge 0$ we have
	$\zeta^{kc_{i0}}=\zeta^{\rho c_{i0}}$ and $\beta_j^{kc_{ij}}=\beta_j^{\rho c_{ij}}\delta_j^{tc_{ij}}$,
	so that
	\[
		f_\rho(\delta_1^{t},\dots,\delta_d^{t})=S_{\rho+qt}\in\Z\qquad\text{for all }t\ge 0,
	\]
	the integrality coming from the hypothesis.
	By rational independence of $\delta_1,\dots,\delta_d$, the set $\set{(\delta_1^{t},\dots,\delta_d^{t}):t\ge 0}$ is dense in $(\mathbb S^1)^d$ \cite[Theorem~4.14]{einsiedler_ward_ergodic_theory}.
	As $f_\rho$ is continuous and integer-valued on this dense set, and $(\mathbb S^1)^d$ is connected while $\Z$ is discrete, $f_\rho$ is constant.
	Denote its value by $a_\rho$. Evaluating at $t=0$ gives $a_\rho=S_\rho$.

	It remains to show $a_\rho=0$ for every $\rho$.
	Distinct characters of $(\mathbb S^1)^d$ are linearly independent, so a constant Laurent polynomial equals its constant coefficient.
	Since every $c_i\neq 0$, each monomial $U_1^{c_{i1}}\cdots U_d^{c_{id}}$ is nonconstant, so the only constant contribution to $f_\rho$ comes from the terms $-x_i$.
	Hence the constant coefficient of $f_\rho$ is $-\sum_{i=1}^{n}x_i$.
	For $\rho=0$ the polynomial is $f_0(U)=\sum_i(U_1^{c_{i1}}\cdots U_d^{c_{id}}-1)x_i$, which is constant with value $f_0(1,\dots,1)=0$, so its constant coefficient vanishes, giving $\sum_{i=1}^{n}x_i=0$.
	Therefore the constant coefficient $-\sum_i x_i$ of every $f_\rho$ is $0$, so $a_\rho=0$ for all $\rho$.

	Finally, for any $k\ge 0$, writing $k=\rho+qt$ yields $S_k=f_\rho(\delta_1^{t},\dots,\delta_d^{t})
	=a_\rho=0$, as claimed.
\end{proof}

The second describes the kernel of a character explicitly, as a finite union of parallel circles.

\begin{lemma}[Kernel of a character]
	\label{lemma:kernel of character}
	Let $h=(a,b)$ be a primitive vector in $\Z^2$ and $n\ge 1$. Writing $h^{\perp}=(-b,a)$,
	\[
		\ker\chi_{nh}
		=\lrset{\frac{j}{n(a^2+b^2)}\,h + th^{\perp}\bmod\Z^2 \ :\ j\in\Z,\ t\in\R}.
	\]
\end{lemma}
\begin{proof}
	Every $x\in\R^2$ is uniquely $x=sh+th^{\perp}$ with $s,t\in\R$, and then $\langle h,x\rangle=s(a^2+b^2)$ because $h\perp h^{\perp}$.
	Now $\chi_{nh}(x)=e^{2\pi i n\langle h,x\rangle}=1$ if and only if $n\langle h,x\rangle\in\Z$.
	Thus $s(a^2+b^2)=j/n$ for some $j\in\Z$, that is, $s=j/(n(a^2+b^2))$.
	Reducing modulo $\Z^2$ gives the stated description (only finitely many values of $j$ give distinct circles).
\end{proof}

\subsection{Combinatorial Preparation}
\label{subsection:combinatorial preparation}
This subsection collects the facts about how finite sets tile $\Z$ and $\Z^2$ that the rest of the section will use.
Only two of them are needed later (Lemma~\ref{lemma:product structure} and Proposition~\ref{proposition:product periodicity}), both in the proof of Theorem~\ref{theorem:p squared one periodic}.

The arguments below are phrased in terms of \define{mask polynomials}.
We define them here.
To a finite set $E\subseteq\Z^d$ we attach the Laurent polynomial
\[
	E(\mathbf x)=\sum_{e\in E}\mathbf x^{e},\qquad \mathbf x^{e}=x_1^{e_1}\cdots x_d^{e_d},
\]
in the variables $x_1,\dots,x_d$, and call it the \define{mask polynomial} of $E$.
We write the same symbol for the set and for its mask polynomial.
In this paper $d=1$ or $d=2$, so the mask polynomial of $E\subseteq\Z$ is $E(x)=\sum_{e\in E}x^{e}$ and that of $E\subseteq\Z^2$ is $E(x,y)=\sum_{(i,j)\in E}x^{i}y^{j}$.

What makes the device useful is the dictionary between set operations and algebra.
A translate $E+v$ has mask polynomial $\mathbf x^{v}E(\mathbf x)$, and a disjoint union $E\sqcup E'$ has mask polynomial $E(\mathbf x)+E'(\mathbf x)$, so conversely any Laurent polynomial with coefficients in $\set{0,1}$ is the mask polynomial of the set of its exponents.
For an infinite set the same expression is a formal Laurent series.
In particular $\one(x)=\sum_{i\in\Z}x^{i}$ is the series of all of $\Z$, and a tiling $E\oplus T=\Z$ is recorded by the identity $E(x)\,T(x)=\one(x)$, a product of a polynomial with a formal series that is well defined coefficientwise.

The following periodicity statement for prime-cardinality tiles of $\Z$ is due to Szegedy
\cite{szegedy_algorithms_to_tile}, originally by a combinatorial argument.
The short algebraic proof we give using the Frobenius identity (the ``freshman's dream'') is that of Horak and Kim \cite[Theorem~16]{horak_kim_algebraic_method} and, independently, of Kari and Szabados \cite[Example~4]{kari_szabados_alg_geom}.
The density clause that turns periodicity into the coset count is elementary.

\begin{lemma}[Periodicity of prime-cardinality tiles of $\Z$]
	\label{lemma:prime tile periodicity}
	Let $p$ be prime and $A\subseteq\Z$ with $\abs A=p$, and suppose $A\oplus T=\Z$.
	Then $T$ is invariant under $p(a-a')$ for all $a,a'\in A$.
	Equivalently, writing $d=\gcd(A-A)$, $T$ is invariant under $p\langle A-A\rangle=pd\,\Z$.
	Consequently $T$ is a union of cosets of $pd\,\Z$, and since $T$ has density $1/p$ it is the union of exactly $d$ of them.
\end{lemma}
\begin{proof}
	The periodicity statement that $T$ is invariant under each $p(a-a')$, hence under $p\langle A-A\rangle=pd\,\Z$, is \cite[Theorem~16]{horak_kim_algebraic_method} (independently \cite[Example~4]{kari_szabados_alg_geom}).
	The argument applies the Frobenius identity $Q(x)^p\equiv Q(x^p)\pmod p$ to the mask polynomial $Q(x)=\sum_{a\in A}x^{a}$. Granting this, $T$ is a union of cosets of $pd\,\Z$, and only the count remains.
	Because $A\oplus T=\Z$ with $\abs A=p$, the set $T$ has density $1/p$: in a long interval the $p$ translates $T-a$, $a\in A$, partition it. If $T$ occupies $k$ of the $pd$ residues modulo $pd$, then its density is $k/(pd)$, so $k/(pd)=1/p$, that is, $k=d$. Thus $T$ is the union of exactly $d$ cosets of $pd\,\Z$.
\end{proof}

\begin{corollary}[Rigidity of prime-cardinality tiles]
	\label{corollary:prime crs rigidity}
	Let $A\subseteq\Z$ with $\abs A=p$ prime and $\gcd(A-A)=1$.
	Then every $T\subseteq\Z$ with $A\oplus T=\Z$ is a coset of $p\Z$.
	In particular the set $T$ with $A\oplus T=\Z$ is unique up to translation.
\end{corollary}
\begin{proof}
	Here $d=\gcd(A-A)=1$, so $pd\,\Z=p\Z$ and Lemma~\ref{lemma:prime tile periodicity} makes $T$ the union of exactly $d=1$ coset of $p\Z$, that is, $T=p\Z+r$.
	(No complete-residue hypothesis is needed: once $T=p\Z+r$, the identity $A\oplus(p\Z+r)=\Z$ with $\abs A=p$ forces $A$ to be a complete residue system modulo $p$ automatically.)
\end{proof}

\begin{lemma}[Divisible sections force a product structure]
	\label{lemma:product structure}
	Let $p$ be a prime and $S\subseteq\Z^2$ a set of size $\abs{S}=p^2$ containing the origin.
	Suppose there are two distinct lines $\ell$ and $m$ through the origin such that every line parallel to $\ell$ or to $m$ meets $S$ in a number of points divisible by $p$.
	Then $\ell$ and $m$ are necessarily \emph{rational} lines.\footnote{%
		A line in $\R^2$ is said to be \define{rational} if it passes through two distinct points of $\Z^2$.}
	Let $u_\ell,u_m\in\Z^2$ be the primitive vectors in the directions of $\ell$ and $m$ respectively.
	\begin{enumerate}[(i)]
		\item \emph{Product structure.} The set $S$ is a product in the basis $(u_\ell,u_m)$: there are $A,B\subseteq\Z$ with $\abs A=\abs B=p$ and
		\[
			S=\set{au_\ell+bu_m:\ a\in A,\ b\in B}.
		\]
		\item \emph{Genuine product.} If $u_\ell,u_m$ form a $\Z$-basis of $\Z^2$ \textup(equivalently $\abs{\det(u_\ell,u_m)}=1$\textup), then the unimodular map $g\in\mathrm{GL}_2(\Z)$ sending $u_\ell\mapsto(1,0)$ and $u_m\mapsto(0,1)$ is an automorphism of $\Z^2$ that carries $S$ onto the product $A\times B$, that is, $g(S)=A\times B$.
		\item \emph{Full affine span.} If $S$ has full affine span, i.e.\ $\langle S-S\rangle=\Z^2$, then $u_\ell,u_m$ automatically form a $\Z$-basis \textup($\abs{\det(u_\ell,u_m)}=1$\textup) and $\gcd(A-A)=\gcd(B-B)=1$. Consequently \textup{(ii)} applies, and $g$ carries $S$ onto the product $A\times B$ with $A,B$ each of full span $\gcd(A-A)=\gcd(B-B)=1$.
	\end{enumerate}
\end{lemma}
\begin{proof}
	\emph{The lines are rational.}
	Since $\abs S=p^2>0$, some line parallel to $\ell$ meets $S$, necessarily in at least $p\ge 2$ points, so that parallel line passes through two lattice points and hence has a primitive integer direction vector $u_\ell\in\Z^2$.
	The line $\ell$ shares this direction and passes through the origin, so it contains the nonzero lattice point $u_\ell$ and is itself rational.
	Likewise for $m$, giving $u_m$.
	The vectors $u_\ell,u_m$ are $\Q$-linearly independent because $\ell\neq m$.

	\begin{enumerate}[(i)]
		\item \emph{Product structure.}
		We first count the grid.
		The lines parallel to $\ell$ partition $\Z^2$ and each meets $S$ in a multiple of $p$.
		As these counts sum to $\abs S=p^2$ and every nonzero one is $\ge p$, at most $p$ of them meet $S$.
		Symmetrically at most $p$ lines parallel to $m$ meet $S$.
		Since $\ell$ and $m$ have distinct directions, each line parallel to $\ell$ meets each line parallel to $m$ in exactly one point, so $S$ is contained in the grid of those intersection points, a set of at most $p\cdot p=p^2$ points.
		As $\abs S=p^2$, equality holds throughout: exactly $p$ lines in each family meet $S$, each in exactly $p$ points, and $S$ is the \emph{full} grid of all $p^2$ intersections.

		We now read off product coordinates from this grid.
		By hypothesis $0\in S$, which we use as base point.
		For $s\in S$ let $s^\flat$ and $s^\sharp$ be the lattice points
		\[
			s^\flat=\lrp{s+\R u_\ell}\cap\R u_m,\qquad
			s^\sharp=\lrp{s+\R u_m}\cap\R u_\ell.
		\]
		Both lie in $S$.
		Because $u_m$ is primitive, the lattice points on the line $\R u_m$ are exactly $\Z u_m$, so $s^\flat\in\Z u_m$.
		Because $s$ and $s^\flat$ are two lattice points on the line $s+\R u_\ell$ of primitive direction $u_\ell$, their difference is an integer multiple of $u_\ell$, so $s-s^\flat\in\Z u_\ell$.
		Writing $s=\alpha u_\ell+\beta u_m$ uniquely over $\Q$, comparison gives $s^\flat=\beta u_m$ and $s-s^\flat=\alpha u_\ell$, whence $\alpha,\beta\in\Z$.
		By the symmetric argument $s^\sharp=\alpha u_\ell$.
		Thus $\alpha$ is constant along each line parallel to $m$ (the level sets $s+\R u_m$), and $\beta$ is constant along each line parallel to $\ell$. Write $a(s)=\alpha$ and $b(s)=\beta$ for these integer coordinates.
		The $p$ lines parallel to $m$ meeting $S$ therefore carry $p$ distinct values of $a$, forming a set $A$ with $\abs A=p$, and similarly the $p$ lines parallel to $\ell$ give $B$ with $\abs B=p$.
		Since $S$ is the full grid, every line parallel to $m$ meets every line parallel to $\ell$ in a point of $S$, so every pair $(a,b)\in A\times B$ occurs, and
		\[
			S=\set{a u_\ell+b u_m:\ a\in A,\ b\in B},
		\]
		which is~(i).

		\item \emph{Genuine product.}
		If $u_\ell,u_m$ form a $\Z$-basis then the unimodular map $a u_\ell+b u_m\mapsto(a,b)$ is exactly the automorphism $g$, and it carries $S$ to the product $A\times B\subseteq\Z^2$.

		\item \emph{Full affine span.}
		Suppose $\langle S-S\rangle=\Z^2$.
		Recall from~(i) that each $s\in S$ has coordinates $s=a(s)u_\ell+b(s)u_m$ with $a(s)\in A$ and $b(s)\in B$.
		Fix a value $b\in B$ and consider the elements of $S$ with that second coordinate.
		By~(i) these are exactly the points $au_\ell+bu_m$ for $a\in A$, all lying on one line parallel to $\ell$.
		Their pairwise differences are the vectors $(a-a')u_\ell$ with $a,a'\in A$, which generate $\gcd(A-A)\,\Z u_\ell$.
		Symmetrically, fixing a value $a\in A$ and letting the second coordinate range over $B$ gives the differences $(b-b')u_m$ with $b,b'\in B$, generating $\gcd(B-B)\,\Z u_m$.
		Every difference in $S-S$ is a sum of one vector of each type, and $u_\ell,u_m$ are independent, so
		\[
			\langle S-S\rangle=\gcd(A-A)\,\Z u_\ell\ \oplus\ \gcd(B-B)\,\Z u_m .
		\]
		This sublattice has index
		\[
			\gcd(A-A)\cdot\gcd(B-B)\cdot\abs{\det(u_\ell,u_m)}
		\]
		in $\Z^2$.\footnote{The factor $\abs{\det(u_\ell,u_m)}$ is the index of $\Z u_\ell\oplus\Z u_m$ in $\Z^2$, and $\gcd(A-A)\cdot\gcd(B-B)$ is the further index of the displayed lattice inside $\Z u_\ell\oplus\Z u_m$.}
		Full affine span forces this index to equal $1$, so all three factors equal $1$.
		In particular $\abs{\det(u_\ell,u_m)}=1$, which is the hypothesis of~(ii), so $g(S)=A\times B$.
	\end{enumerate}
	This finishes the proof.
\end{proof}

\begin{lemma}[Fibre decomposition of a product tiling]
	\label{lemma:product fibre decomposition}
	Let $A,B\subseteq\Z$ be finite and $F=A\times B$, and let $T$ be an $F$-tiling of $\Z^2$.
	Write the row and column fibres 
	$$
		S_y=\set{x:(x,y)\in T}
		\quad
		\text{and}
		\quad
		T_x=\set{y:(x,y)\in T}.
	$$
	Then, for all $n,m\in\Z$:
	\begin{enumerate}[(1)]
		\item the sets $S_{n-b}$ ($b\in B$) are pairwise disjoint, and $A\oplus R_n=\Z$, where $R_n=\bigsqcup_{b\in B}S_{n-b}$.
		\item the sets $T_{m-a}$ ($a\in A$) are pairwise disjoint, and $B\oplus C_m=\Z$, where $C_m=\bigsqcup_{a\in A}T_{m-a}$.
	\end{enumerate}
\end{lemma}
\begin{proof}
	We only prove (1) as statement (2) is identical with the roles of the coordinates exchanged.
	The argument is algebraic.
	To the finite sets $A$ and $B$ attach the mask polynomials
	\[
		A(x)=\sum_{a\in A}x^a,\qquad B(y)=\sum_{b\in B}y^b,
	\]
	so that, since $F=A\times B$, the Laurent polynomial of $F$ is
	\[
		F(x,y)=\sum_{(i,j)\in F}x^iy^j=A(x)\,B(y).
	\]
	Write $\one(x)=\sum_{i\in\Z}x^i$ for the all-ones series and $S_m(x)=\sum_{u\in S_m}x^u$ for the row-fibre polynomials, so that the Laurent series of $T$ is
	\[
		T(x,y)=\sum_{(i,j)\in T}x^iy^j=\sum_{m\in\Z}S_m(x)\,y^m.
	\]
	Every product below pairs a Laurent polynomial with a formal series, hence is well defined.
	That $T$ is an $F$-tiling (each cell of $\Z^2$ covered exactly once) is nothing but the identity
	\[
		F(x,y)\,T(x,y)=\one(x)\,\one(y),
		\qquad\text{that is,}\qquad
		A(x)\,B(y)\,T(x,y)=\one(x)\,\one(y).
	\]
	Using
	\[
		B(y)\,T(x,y)=\sum_{n\in \Z}\lrp{\sum_{b\in B}S_{n-b}(x)}y^n
	\]
	and comparing coefficients of $y^n$ gives
	\[
		A(x)\,\rho_n(x)=\one(x),\qquad\text{where }\rho_n(x)=\sum_{b\in B}S_{n-b}(x).
	\]
	The coefficients of $A(x)$ and of $\rho_n(x)$ are non-negative integers, and every coefficient of the product is $1$.
	If the coefficient of $x^k$ in $\rho_n(x)$ were at least $2$ for some $k$, then for any $a\in A$ the coefficient of $x^{k+a}$ in $A(x)\rho_n(x)$ would also be at least $2$, a contradiction.
	Hence every coefficient of $\rho_n(x)$ lies in $\set{0,1}$, so the sets $S_{n-b}$ ($b\in B$) are pairwise disjoint and $\rho_n(x)$ is the mask polynomial of $R_n=\bigsqcup_{b\in B}S_{n-b}$.
	The identity $A(x)\,R_n(x)=\one(x)$ then says exactly that $A\oplus R_n=\Z$.
\end{proof}

We isolate a technical lemma below that will be used in the proof of the proposition that follows. 

\begin{lemma}
	\label{lemma:residue block rigidity}
	Let $p$ be prime and $B\subseteq\Z$ a complete residue system modulo $p$ with $\gcd(B-B)=1$.
	For any integer $y$ let $\bar y$ denote the residue of $y$ modulo $p$.
	Let $I\subseteq\set{0,1,\dots,p-1}$ with $\abs I\ge 2$, and let $r\colon\Z\to X$ be a function into any set $X$ such that
	\begin{equation}
		\label{equation:block constancy}
		r_{y+b}=r_{y+c}\qquad\text{for all }b,c\in B\text{ and all }y\in\Z\text{ with }\bar y\in I.
	\end{equation}
	Then $r$ is constant.
\end{lemma}
\begin{proof}
	For $y\in\Z$ with $\bar y\in I$, condition~\eqref{equation:block constancy} makes $r_{y+b}$ independent of $b\in B$.
	Write $\sigma(y)$ for this common value, so that $r_{y+b}=\sigma(y)$ for all $b\in B$ and integers $y$ such that $\bar y\in I$.
	This defines the function $\sigma$ on $J:=\set{y\in\Z:\bar y\in I}$.

	\smallskip\noindent\emph{Block decomposition: for every $n\in\Z$ and every $i\in I$ there is a unique $b\in B$ for which $y:=n-b$ satisfies $\bar y=i$.}
	Since $B$ is a complete residue system modulo $p$, there is a unique $b\in B$ with $b\equiv n-i\pmod p$.
	For this $b$ the integer $y=n-b$ satisfies $y\equiv i\pmod p$, that is, $\bar y=i$.

	\smallskip
	Fixing one $i\in I$, the block decomposition writes every $n\in\Z$ as $n=y+b$ with $b\in B$ and $y\in J$, whence $r_n=r_{y+b}=\sigma(y)$.
	Thus $r$ is determined by $\sigma$, and it suffices to show that $\sigma$ is constant on $J$.

	Letting $i$ instead range over $I$ records a first constraint on $\sigma$:
	\begin{equation}
		\label{equation:sigma step}
		\sigma(y)=\sigma(y')\qquad\text{whenever }\bar y,\bar y'\in I\text{ and }y-y'\in B-B.
	\end{equation}
	Indeed, write $y-y'=b_i-b_j$ with $b_i,b_j\in B$ and set $w=y+b_j = y' + b_i$.
	Now $\bar y\in I$ and $b_j\in B$, so the definition of $\sigma$ gives $\sigma(y)=r_{y+b_j}=r_w$.
	Similarly, $\bar y'\in I$ and $b_i\in B$ give $\sigma(y')=r_{y'+b_i}=r_w$.
	Hence $\sigma(y)=\sigma(y')$.
	
	This is weaker than constancy.
	It equates $\sigma$ only across a single $(B-B)$-step joining two points of $J$, and such a step is nontrivial only when $\abs I\ge2$ (for $\abs I=1$ the two endpoints coincide, and the Remark after the proof exhibits a non-constant $\sigma$).
	The rest of the proof shows that when $\abs I\ge2$ these steps already generate enough of $\Z$ to force $\sigma$ to a single value.

	Since $\abs I\ge 2$, choose distinct $q,q'\in I$ and set $\delta=q'-q$, so $\delta\not\equiv0\pmod p$.
	For $b\in B$ let $b^{+}\in B$ be the unique element with $b^{+}\equiv b+\delta\pmod p$, and put $h_b=b^{+}-b\in B-B$, so $h_b\equiv\delta\pmod p$.
	Since $p$ is prime and $\delta\not\equiv0\pmod{p}$, the map $b\mapsto b^{+}$ is a single $p$-cycle on $B$.
	List $B$ along it as $b_0,b_1,\dots,b_{p-1}$ with $b_{k+1}=b_k^{+}$ (indices modulo $p$), so $h_{b_k}=b_{k+1}-b_k$.

	\smallskip\noindent\emph{Claim 1: the elements $h_b$ generate $\Z$.}
	Because $b\mapsto b^{+}$ merely permutes $B$, we have
	\[
		\sum_{b\in B} h_b
		=\sum_{b\in B} b^{+}-\sum_{b\in B} b=0.
	\]
	Listing the $h_b$ along the cycle, the partial sums telescope:
	\[
		\sum_{l=0}^{k-1}h_{b_l}=\sum_{l=0}^{k-1}(b_{l+1}-b_l)=b_k-b_0 .
	\]
	As $k$ runs through $0,1,\dots,p-1$ the element $b_k$ runs through all of $B$, so every difference $b_k-b_0$ lies in $\langle h_b:b\in B\rangle$.
	Since $b_i-b_j=(b_i-b_0)-(b_j-b_0)$, every element of $B-B$ lies there too, and as each $h_b\in B-B$ this gives $\langle h_b:b\in B\rangle=\langle B-B\rangle$.
	Finally $\langle B-B\rangle=\gcd(B-B)\,\Z=\Z$ since $\gcd(B-B)=1$, establishing the claim.

	\smallskip\noindent\emph{Claim 2: the differences $h_b-h_c$ generate $p\Z$.}
	Let $H=\langle h_b-h_c:b,c\in B\rangle$.
	Every $h_b\equiv\delta\pmod p$, so each $h_b-h_c$ is divisible by $p$, giving $H\subseteq p\Z$.
	For the reverse inclusion, note first that $ph_{b_0}\in H$: using $\sum_b h_b=0$ and $\abs B=p$,
	\[
		ph_{b_0}=\sum_{b\in B} h_{b_0} - \sum_{b\in B} h_b=\sum_{b\in B} (h_{b_0}-h_b)\in H .
	\]
	Pass to the quotient $\Z/H$.
	The $h_b$ are pairwise congruent modulo $H$, so they share a common image $\bar h\in\Z/H$.
	By Claim~1 the $h_b$ generate $\Z$, hence $\bar h$ generates $\Z/H$, so $\Z/H$ is cyclic.
	Moreover $p\bar h=\overline{ph_{b_0}}=0$, so the order of $\bar h$ divides $p$, whence $[\Z:H]\in\set{1,p}$.
	But $H\subseteq p\Z$ forces $[\Z:H]\ge p$, so $[\Z:H]=p$ and $H=p\Z$.

	\smallskip\noindent\emph{Claim 3: $\sigma$ is constant.}
	Consider the shifts preserving $\sigma$ on the points $z$ with $\bar z=q$,
	\[
		G=\set{t\in p\Z:\ \sigma(z)=\sigma(z+t)\text{ for every }z\in\Z\text{ with }\bar z=q}.
	\]
	This is well-defined, since for $t\in p\Z$ and $\bar z=q$ the shifted point satisfies $\overline{z+t}=q$, so both $\sigma$-values are defined.
	Moreover $G$ is a subgroup of $p\Z$.\footnote{%
		Indeed, $G$ contains $0$.
		If $t\in G$, then applying its defining identity at $z-t$ gives $-t\in G$.
		If $t,t'\in G$, then $\sigma(z)=\sigma(z+t)=\sigma(z+t+t')$ for every $z$ with $\bar z=q$, so $t+t'\in G$.}
	Each generator $h_b-h_c$ of $p\Z$ lies in $G$: for $\bar z=q$ the two-step path
	\[
		z\to z+h_b\to z+h_b-h_c
	\]
	has $\overline{z+h_b}=q'$ at the midpoint and $\overline{z+h_b-h_c}=q$ at the endpoint, both in $I$, and each leg is a $(B-B)$-step, so~\eqref{equation:sigma step} gives
	\[
		\sigma(z)=\sigma(z+h_b)=\sigma(z+(h_b-h_c)).
	\]
	By Claim~2 these differences generate $p\Z$, so $G=p\Z$, and $\sigma$ takes a single value $v$ on every $z$ with $\bar z=q$.
	Finally, any $y$ with $\bar y\in I$ satisfies $\sigma(y)=v$: choosing $d\in B-B$ with $d\equiv q-\bar y\pmod p$ (possible since $B-B$ meets every residue class), the integer $y+d$ satisfies $\overline{y+d}=q$, and~\eqref{equation:sigma step} gives $\sigma(y)=\sigma(y+d)=v$.
	Thus $\sigma(y)=v$ for every $y\in J$, and by the opening paragraph $r_n=v$ for every $n\in\Z$, so $r$ is constant.
\end{proof}

\begin{remark}
	The hypothesis $\abs I\ge 2$ cannot be relaxed to $I\neq\varnothing$.
	For $p=2$, $B=\set{0,1}$ and $I=\set 0$, condition~\eqref{equation:block constancy} reads $r_{2k}=r_{2k+1}$, satisfied by the non-constant $r$ with $r_{2k}=r_{2k+1}=k\bmod 2$.
\end{remark}

\begin{proposition}[Axis-periodicity of product tilings]
	\label{proposition:product periodicity}
	Let $p$ be prime and let $A,B\subseteq\Z$ be complete residue systems modulo $p$ with $\gcd(A-A)=\gcd(B-B)=1$.
	Then every $(A\times B)$-tiling $T$ of $\Z^2$ has a coordinate period.
	More precisely, either $(0,p)+T=T$ or $(p,0)+T=T$ (or both).
\end{proposition}
\begin{proof}
	Throughout, 
	$$
		\one(x)=\sum_{i\in\Z}x^i
		\quad \text{ and } \quad
		\one_p(x)=\sum_{k\in\Z}x^{pk},
	$$
	and similarly in $y$.
	Every product below pairs a Laurent polynomial with a formal series, hence is well defined.

	Since $A$ is a complete residue system modulo $p$, the map $A\times\Z\to\Z$, $(a,k)\mapsto a+pk$, is a bijection, so $A(x)\,\one_p(x)=\one(x)$, and similarly $B(y)\,\one_p(y)=\one(y)$.
	We use this in the final step below.
	Put
	$$
		U(x, y)=(x^p-1)T(x, y)
		\quad \text{ and } \quad
		W(x, y)=(y^p-1)T(x, y),
	$$
	so that
	\[
		U=0\iff (p,0)+T=T,
		\qquad  \text{ and } \qquad
		W=0\iff (0,p)+T=T.
	\]
	The proposition asserts $U=0$ or $W=0$. 
	
	As in Lemma~\ref{lemma:product fibre decomposition} define
	$$
		S_y=\set{x:(x,y)\in T}
		\quad
		\text{and}
		\quad
		T_x=\set{y:(x,y)\in T}.
	$$
	We know from Lemma~\ref{lemma:product fibre decomposition} that the sets $S_{n-b}$ are pairwise disjoint as $b$ varies over $B$, and that the set $R_n := \bigsqcup_{b\in B} S_{n-b}$ is an $A$-tiling of $\Z$.
	Using the fact that $A$ is a cluster of prime size with $\gcd(A - A) = 1$, we deduce from Corollary~\ref{corollary:prime crs rigidity} that each $R_n$ is a coset of $p\Z$.
	Thus there is a sequence $(r_n)_{n\in \Z}$ valued in $\set{0, 1, \ldots, p-1}$ such that $R_n = p\Z+r_n$ for all $n$.
	By the same reasoning we also get that the sets $T_{m-a}$ are pairwise disjoint as $a$ varies over $A$, that each $C_m=\bigsqcup_{a\in A}T_{m-a}$ is a $B$-tiling of $\Z$, and that there is a sequence $(s_m)_{m\in \Z}$ valued in $\set{0, 1, \ldots, p-1}$ such that $C_m = p\Z+s_m$ for all $m$.
	In particular 
	$$
		B(y)T(x, y)
		=\sum_n R_n(x)y^n
		=\one_p(x)\sum_n x^{r_n}y^n.
	$$
	As $(x^p-1)\one_p(x)=0$, multiplying by $(x^p-1)$ in the above equation gives $B(y)U(x, y)=0$, and symmetrically $A(x)W(x, y)=0$.

	\smallskip\noindent\emph{We show that the constancy of $r$ or $s$ suffices to show periodicity.}
	Suppose $r\equiv\rho$ is constant.
	Translating $T$ by $(-\rho,0)$ yields another $(A\times B)$-tiling, replaces each row residue $r_n$ by $r_n-\rho$, and commutes with translation by $(0,p)$, hence leaves the relation $(0,p)+T=T$ unchanged.
	We may therefore replace $T$ by this translate and assume $\rho=0$, so that $R_n=p\Z$ for every $n$.
	Then $R_n=\bigsqcup_b S_{n-b}\subseteq p\Z$ for every $n$, so every anchor $(x,y)\in T$ has $x\equiv0\pmod p$. Thus $T\subseteq p\Z\times\Z$.
	The defect $W(x, y)=(y^p-1)T(x, y)$ is then supported on $p\Z\times\Z$ as well, so we may write 
	$$
		W(x,y)=\sum_{k\in\Z}x^{pk}\,w_k(y)
	$$
	with each $w_k$ a Laurent series in $y$.
	Recall that $A(x)W(x, y)=0$.
	Because $A$ is a complete residue system modulo $p$, the exponents $a+pk$ \textup($a\in A$, $k\in\Z$\textup) are pairwise distinct, so in
	\[
		A(x)W(x, y)=\sum_{a\in A}\sum_{k\in\Z}x^{a+pk}\,w_k(y)
	\]
	the coefficient of $x^{a+pk}$ is exactly $w_k(y)$. Its vanishing forces $w_k=0$ for every $k$, so $W=0$, that is $(0,p)+T=T$.
	Symmetrically, when $s$ is constant the same argument with the coordinates exchanged gives $U=0$ and $(p,0)+T=T$.
	It therefore suffices to prove that $r$ or $s$ is constant.

	If $s$ is constant, then $(p,0)+T=T$ as shown above and we are done. Assume henceforth that $s$ is non-constant, and set $I=\Image(s)$, so that $\abs I\ge 2$.
	We check the hypothesis of Lemma~\ref{lemma:residue block rigidity} for the sequence $r=(r_n)_{n\in\Z}$ with this $I$.
	For a nonempty row $y$ and any $b\in B$ we have $S_y\subseteq R_{y+b}=p\Z+r_{y+b}$. 
	Since cosets of $p\Z$ are disjoint, the residues $r_{y+b}$ coincide as $b$ ranges over $B$.
	Thus $r_{y+b}=r_{y+c}$ for all $b,c\in B$ whenever the row $y$ is nonempty.
	Moreover a row $y$ is nonempty if and only if $\bar y\in\Image(s)$, where $\bar y$ is the residue of $y$ modulo $p$: if $(x,y)\in T$, then for any $a\in A$ and $m=x+a$ we have 
	$$
		y\in T_x\subseteq C_m=p\Z+s_m,
	$$
	so $\bar y=s_m\in\Image(s)$.
	Conversely $\bar y=s_m$ gives $y\in C_m=\bigsqcup_{a\in A}T_{m-a}$, so $(m-a,y)\in T$ for some $a\in A$.
	Combining the two, $r_{y+b}=r_{y+c}$ for all $b,c\in B$ and all $y$ with $\bar y\in I$.
	As $\gcd(B-B)=1$ and $\abs I\ge 2$, Lemma~\ref{lemma:residue block rigidity} gives that $r$ is constant, and the case of constant $r$ above yields $(0,p)+T=T$.
	In all cases $T$ has a coordinate period.
\end{proof}

\subsection{The Spectral Theorem}

Let $\T^2=(\R/\Z)^2$ be the 2-torus, and recall that the characters of $\T^2$ are indexed by $\Z^2$: for each $g\in \Z^2$, the map $\chi_g:\T^2\to \mathbb S^1$ defined as 
$$
	\chi_g(x)=e^{2\pi i\langle g,x\rangle}
$$
is a character, and $\set{\chi_g:\ g\in \Z^2}$ is precisely the set of characters of $\T^2$.

Now, for any probability measure $\nu$ on $\T^2$, the group $\Z^2$ acts on $L^2(\T^2,\nu)$ by multiplication: 
$$
	\sigma_g(\varphi)=\chi_g\,\varphi
$$
for all $g\in \Z^2$ and $\vp\in L^2(\T^2, \nu)$.
The constant function $\mathbf{1}$ is cyclic for this action, and the spectral theorem says this is the universal example.

Let $\mc H$ be a Hilbert space.
Recall that a vector $v\in\mathcal H$ is \emph{cyclic} for a unitary representation $\tau\colon\Z^2\to\mr U(\mathcal H)$ if the linear span of its orbit $\set{\tau_g v:g\in\Z^2}$ is dense in $\mathcal H$.
Equivalently, the only closed $\tau$-invariant subspace containing $v$ is $\mathcal H$ itself.
In the multiplication example above the orbit of $\mathbf 1$ is the set of characters $\set{\chi_g:g\in\Z^2}$, whose linear span is dense in $L^2(\T^2,\nu)$, so $\mathbf 1$ is cyclic.

\begin{theorem}[Spectral theorem]
	\label{theorem:spectral theorem}
	Let $\tau\colon\Z^2\to \mr U(\mathcal H)$ be a unitary representation with a cyclic unit vector $v$.
	Then there is a unique probability measure $\nu$ on $\T^2$ and a unitary isomorphism $\theta\colon\mathcal H\to L^2(\T^2,\nu)$ with $\theta(v)=\mathbf{1}$ that intertwines each $\tau_g$ with multiplication by $\chi_g$.
	Explicitly, writing $M_{\chi_g}\colon L^2(\T^2,\nu)\to L^2(\T^2,\nu)$ for the multiplication operator $\varphi\mapsto\chi_g\,\varphi$, we have
	\[
		\theta\circ\tau_g=M_{\chi_g}\circ\theta\qquad\text{for every }g\in\Z^2;
	\]
	that is, the square
	\[
		\begin{tikzcd}[column sep=large, row sep=large]
			\mathcal H \arrow[r, "\tau_g"] \arrow[d, "\theta"'] & \mathcal H \arrow[d, "\theta"] \\
			L^2(\T^2,\nu) \arrow[r, "M_{\chi_g}"'] & L^2(\T^2,\nu)
		\end{tikzcd}
	\]
	commutes for every $g\in\Z^2$.
	The probability measure $\nu$ is called the \textbf{spectral measure} of $v$.
\end{theorem}

\subsection{Dynamical Formulation}
\label{subsection:dynamical formulation}

The following notation is fixed for the remainder of this subsection. Let $F\subseteq\Z^2$ be an exact cluster and fix an $F$-tiling $T$.
Then $1_T$ is an element of the full shift $\set{0, 1}^{\Z^2}$.
Let $X$ be the orbit closure of $1_T$ in this full shift, and equip $X$ with a $\Z^2$-ergodic measure $\mu$.\footnote{
	Such a measure $\mu$ exists.
	The space $X$ is a nonempty compact space on which $\Z^2$ acts by homeomorphisms, so it carries a $\Z^2$-invariant Borel probability measure (Krylov--Bogolyubov for the amenable group $\Z^2$), and the ergodic decomposition then provides an ergodic one.}
Set
\[
	A=\set{x\in X : x(0,0)=1},\qquad f=1_A\in L^2(X,\mu).
\]
Since $|F|$ many translates of $A$ partition $X$, we have $\mu(A) = 1/|F|$.
Also, $f\neq0$ since 
$$
	\norm f_2^2=\mu(A)=1/\abs F>0,
$$
so that the cyclic unit vector $f/\norm f_2$ is well defined.
Let
$$
	\mc H = \overline{\text{Span}\set{g\cdot f/\|f\|_2:\ g\in \Z^2}}.
$$
Then $f/\norm{f}_2$ is a cyclic unit vector of $\mc H$.
Let $\nu$ be the spectral measure associated to this cyclic vector and $\theta:\mc H\to L^2(\T^2, \nu)$ be the intertwining isomorphism furnished by the spectral theorem.

We will use the following theorem, which transfers periodicity from the function $f$ back to the tilings themselves.

\begin{theorem}[{\cite[Section~2]{bhattacharya_tilings}}]
	\label{theorem:bhattacharya correspondence}
	If $A$ is $1$-periodic, that is, if there is a nonzero vector $g\in \Z^2$ such that $g\cdot A = A$ (up to sets of measure 0), then $\mu$-almost every point of $X$ is $1$-periodic.
\end{theorem}
\begin{proof}
	We show that $g$ itself is a period of $\mu$-almost every point. For $v\in\Z^2$ the cylinder $\set{x:x(v)=1}$ equals the translate $v\cdot A$, so applying the measure-preserving, commuting transformation $v\cdot{}$ to the hypothesis $g\cdot A=A$ gives
	\[
		\set{x:x(v+g)=1}=(v+g)\cdot A=v\cdot(g\cdot A)=v\cdot A=\set{x:x(v)=1}\qquad(\text{mod }\mu).
	\]
	Hence $\mu\set{x:x(v+g)\neq x(v)}=0$ for each $v$, and as $\Z^2$ is countable the union of these null sets is null. So for $\mu$-almost every $x$ we have $g\cdot x=x$, and since $g\neq0$ such an $x$ is $1$-periodic.
\end{proof}

We will also use the dilation lemma, which says a tiling by $F$ is also a tiling by $\alpha F$ whenever $\alpha$ is coprime to $\abs F$.
Several authors discovered it independently. 

\begin{theorem}[Dilation lemma {\cite[Corollary~11]{horak_kim_algebraic_method}}]
	\label{theorem:dilation lemma}
	Let $F\subseteq\Z^2$.
	If $T$ is an $F$-tiling, then $T$ is also an $\alpha F$-tiling for every $\alpha$ coprime to $\abs F$, where $\alpha F=\set{\alpha a : a\in F}$.
\end{theorem}

Transporting the tiling identity through the spectral isomorphism, the dilation lemma forces the spectral measure $\nu$ to satisfy a family of trigonometric identities, indexed by the dilations $\alpha$.
These are the equations analysed in the proof of Theorem~\ref{theorem:p squared one periodic} below.

\begin{lemma}
	\label{lemma:vanishing equations}
	For $\nu$-almost every $\xi\in\T^2\setminus\set 0$, the relation
	\begin{equation}
		\label{equation:vanishing equations}
		\sum_{g\in F}\chi_{\alpha g}(\xi)=0
	\end{equation}
	holds for every integer $\alpha$ coprime to $\abs F$.
\end{lemma}
\begin{proof}
	Fix an integer $\alpha$ coprime to $\abs F$.
	Every $x\in X$ is an $F$-tiling, so by the dilation lemma~\ref{theorem:dilation lemma} it is also an
	$\alpha F$-tiling. In terms of $f=1_A$ this property is the identity
	\[
		\sum_{g\in F}(\alpha g)\cdot f=\mathbf 1_X\qquad\text{in }L^2(X,\mu),
	\]
	where $\mathbf 1_X$ is the constant function $1$ and $g\cdot{}$ is the translation action of $\Z^2$ on
	$L^2(X,\mu)$: evaluated at a tiling $x=1_T\in X$, the left-hand side counts the pairs
	$(g,t)\in F\times T$ with $\alpha g+t=0$ (the ways the origin is covered), which is exactly one
	because $\alpha F\oplus T=\Z^2$.

	Apply the isomorphism $\theta\colon\mc H\to L^2(\T^2,\nu)$ of
	Theorem~\ref{theorem:spectral theorem}. The constant $\mathbf 1_X$ lies in the cyclic subspace
	$\mc H$, being the finite sum of translates of $f$ just displayed. Set $\psi=\theta(\mathbf 1_X)$.
	Since $\theta(f)=\norm f_2\,\mathbf 1$ and $\theta$ intertwines $g\cdot{}$ with multiplication by
	$\chi_g$,
	\[
		\norm f_2\sum_{g\in F}\chi_{\alpha g}
		=\theta\Bigl(\sum_{g\in F}(\alpha g)\cdot f\Bigr)=\theta(\mathbf 1_X)=\psi .
	\]
	Now $\mathbf 1_X$ is fixed by every translation, $h\cdot\mathbf 1_X=\mathbf 1_X$, so applying $\theta$
	gives $\chi_h\,\psi=\psi$ in $L^2(\T^2,\nu)$ for every $h\in\Z^2$. That is, $(\chi_h-1)\psi=0$
	$\nu$-almost everywhere. For each $\xi\neq 0$ some $h\in\Z^2$ has $\chi_h(\xi)\neq 1$, so $\psi$
	vanishes $\nu$-almost everywhere on $\T^2\setminus\set 0$. Therefore
	\[
		\sum_{g\in F}\chi_{\alpha g}(\xi)=\norm f_2^{-1}\psi(\xi)=0\qquad\text{for $\nu$-almost every }\xi\neq 0 .
	\]
	Intersecting these $\nu$-conull sets over the countably many $\alpha$ coprime to $\abs F$ leaves a single
	$\nu$-conull set off the origin on which all the equations hold.
\end{proof}

The one number-theoretic input we need is the R\'edei--de Bruijn--Schoenberg theorem on vanishing sums of prime-power roots of unity: a sum of $p^m$-th roots of unity can vanish only by being built up, with multiplicity, from full sets of $p$-th roots.

The use of cyclotomic-polynomial divisibility in tiling theory is not new.
In dimension one it goes back to Tijdeman \cite{tijdeman_decomposition_of} and was later used by Coven and Meyerowitz \cite{coven_meyerowitz_tiling_integers}.
A more recent development is by {\L}aba and Londner \cite{laba_londner_three_primes}.

\begin{lemma}[Vanishing sums of prime-power roots of unity; the prime-power case of the R\'edei--de Bruijn--Schoenberg theorem {\cite[Theorem~2.2]{lam_leung_vanishing_sums}}]
	\label{lemma:prime power vanishing sums}
	Let $p$ be a prime, let $m\ge 1$, and let $\zeta$ be a primitive $p^m$-th root of unity.
	For a finite multiset $S$ of integers attach a polynomial $P(z)=\sum_{a\in S}z^{a}$.
	The following are equivalent:
	\begin{enumerate}[(a)]
		\item $P(\zeta)=0$,
		\item $\Phi_{p^m}(z)=\Phi_p\bigl(z^{p^{m-1}}\bigr)$ divides $P(z)$ in $\Z[z]$,\footnote{%
				If $S$ contains negative integers, then $P$ is a Laurent polynomial.
				In that case choose any integer $T$ with $z^{T}P(z)\in\Z[z]$ and read~(b) as $\Phi_{p^m}(z)\mid z^{T}P(z)$ in $\Z[z]$.
				Since $\Phi_{p^m}(0)=1$ makes $\Phi_{p^m}$ coprime to $z$, this is independent of the choice of $T$.
				This reading agrees with divisibility in the Laurent ring: because the clearing monomials $z^{T}$ are units of $\Z[z,z^{-1}]$, we have $\Phi_{p^m}\mid P$ in $\Z[z,z^{-1}]$ if and only if $\Phi_{p^m}\mid z^{T}P$ in $\Z[z]$ for one, equivalently every, valid $T$.
			}
		\item the multiplicity function $r\mapsto |\set{a\in S : a\equiv r \pmod{p^m}}|$ is constant on every coset of the order-$p$ subgroup $p^{m-1}\Z/p^m\Z$ of $\Z/p^m\Z$.\footnote{%
				The count is with multiplicity: an element of the multiset $S$ occurring $j$ times contributes $j$.
				The multiplicity function is defined on integers, but if $r\equiv r'\pmod{p^m}$ then the two counts agree.
				Thus it depends only on the class $r+p^m\Z$ and descends to a well-defined function on $\Z/p^m\Z$, sending a class to the number of elements of $S$ in it.
				It is this induced function on $\Z/p^m\Z$ that condition~(c) requires to be constant on each coset.}
	\end{enumerate}
	When they hold, $p\mid\abs S$.
\end{lemma}

A proof is given in Appendix~\ref{appendix:vanishing sums} for completeness.

\medskip \noindent
Before we proceed we state some definitions.
A \define{section} of $F$ along a line $\ell$ is the intersection $F\cap\ell$, and for a nonzero $g\in\Z^2$ the \define{sections of $F$ in the direction $g$} are the sets $F\cap\ell$ as $\ell$ ranges over the lines parallel to $g$.
We say that $F$ has \define{$p$-divisible sections} in the direction $g$ if $p\mid\abs{F\cap\ell}$ for every line $\ell$ parallel to $g$.
With this language, the next lemma says that a direction along which the vanishing equations~\eqref{equation:vanishing equations} have infinitely many solutions on a single kernel is one in which $F$ has only $p$-divisible sections.

\begin{lemma}[From infinitely many solutions to divisible sections]
	\label{lemma:divisible sections}
	Assume $F$ contains the origin, and let $g_0$ be a nonzero vector in $\Z^2$.
	Write $g_0=nh$ with $h\in\Z^2$ primitive.
	Suppose there are infinitely many points $\xi\in\ker\chi_{g_0}$ at which the vanishing equations~\eqref{equation:vanishing equations} hold, that is,
	\[
		\sum_{g\in F}\chi_{\alpha g}(\xi)=0\qquad\text{for every integer }\alpha\text{ coprime to }p.
	\]
	Then there is an integer $m\ge 1$, such that every line $\ell$ parallel to $g_0$ and its section $F_\ell=F\cap\ell$ satisfy the following.
	\begin{enumerate}[(i)]
		\item \emph{Structural form.} The Laurent polynomial $\sum_{g\in F_\ell}z^{\langle g,h\rangle}$ is divisible by
			$$
				\Phi_{p^m}(z)=\Phi_p\bigl(z^{p^{m-1}}\bigr).
			$$
			Equivalently, by Lemma~\ref{lemma:prime power vanishing sums}, the multiplicity function 
			$$
				r\mapsto
				|\set{g\in F_\ell:\ \langle g,h\rangle\equiv r\ (\mathrm{mod}\ p^m)}|
			$$
			is constant on every coset of the order-$p$ subgroup $p^{m-1}\Z/p^m\Z$ of $\Z/p^m\Z$.
		\item \emph{Divisibility.} In particular $p\mid\abs{F\cap\ell}$.
	\end{enumerate}
	\end{lemma}
\begin{proof}
	The summary of the proof is this: restrict the equations to one of the parallel lines making up $\ker\chi_{g_0}$, turn them into a one-variable Laurent polynomial with infinitely many roots in $\mathbb S^1$, and extract a cyclotomic divisibility of each section's mask polynomial.
	This divisibility is the structural form~(i).
	Evaluating it at $1$ then yields $p\mid\abs{F\cap\ell}$ in~(ii).
	Here are the details.

	Write $g_0=nh$ with $h=(a,b)$ primitive, and put $v=h^{\perp}=(-b,a)$, the direction transverse to $g_0$.
	By Lemma~\ref{lemma:kernel of character}, $\ker\chi_{g_0}$ is a finite union of lines $\set{rh+tv :\ t\in\R}$ indexed by rationals $r=j/(n(a^2+b^2))$.
	Infinitely many solutions of \eqref{equation:vanishing equations} must lie on one such line.
	So fix a rational $r=c/d$ with $\gcd(c,d)=1$ for which infinitely many $t\in[0,1)$ satisfy, for every $\alpha$ coprime to $p$,
	\[
		\sum_{g\in F}\chi_{\alpha g}(r h+t v)
		=\sum_{g\in F} e^{2\pi i\,\alpha r\langle g,h\rangle}\,e^{2\pi i\,\alpha\langle g,v\rangle t}=0 .
	\]
	For fixed $\alpha$ this is a Laurent polynomial in $z=e^{2\pi i t}$, namely $\sum_{g\in F}e^{2\pi i\,\alpha r\langle g,h\rangle}z^{\alpha\langle g,v\rangle}$, with infinitely many roots on $\mathbb S^1$, hence identically zero.

	Let $p^m$ be the highest power of $p$ dividing $d$, and let $\beta = d/p^m$.
	Then $\beta$ is coprime to $p$ with $d/\beta = p^m$.
	Setting
	$\alpha=\beta$ and using $\beta r=c/p^m$, the polynomial
	\[
		\sum_{g\in F}e^{2\pi i\,\langle g,h\rangle c/p^{m}}\,z^{\beta\langle g,v\rangle}
	\]
	is identically zero. Group $F$ by the value of $\langle g,v\rangle$ into classes $F_1,\dots,F_l$.
	Since $v\perp g_0$, each $F_j$ is exactly the intersection of $F$ with a line parallel to $g_0$. The
	coefficient on the monomial attached to $F_j$ must vanish:
	\begin{equation}
		\label{equation:section vanishing}
		\sum_{g\in F_j}e^{2\pi i\,\langle g,h\rangle c/p^{m}}=0,\qquad j=1,\dots,l .
	\end{equation}
	If $m=0$ this sum is $\abs{F_j}\neq 0$, a contradiction.
	So $m\ge 1$, that is, $p\mid d$, and as $\gcd(c,d)=1$ we also have $p\nmid c$.
	Let $\zeta=e^{2\pi i/p^{m}}$, a primitive $p^m$-th root of unity.
	The vanishing in~\eqref{equation:section vanishing} says $\zeta$ is a root of
	$$
		Q_j(z):=\sum_{g\in F_j}z^{\langle g,h\rangle c},
	$$
	Setting $P_j(w):=\sum_{g\in F_j}w^{\langle g,h\rangle}$, we have $Q_j(z)=P_j(z^{c})$ identically, so $Q_j(\zeta)=P_j(\zeta^{c})$.
	Thus $\zeta^{c}$ is a root of $P_j$.
	Since the exponents $\langle g,h\rangle$ may be negative, $P_j$ is a priori a Laurent polynomial.
	Choose $N$ with $N+\langle g,h\rangle\ge 0$ for all $g\in F_j$, so that 
	$$
		\widetilde P_j(w):=w^{N}P_j(w)=\sum_{g\in F_j}w^{N+\langle g,h\rangle}\in\Z[w]
	$$ 
	is an ordinary polynomial with the same nonzero roots as $P_j$.
	Since $p\nmid c$, the power $\zeta^{c}$ is again a primitive $p^m$-th root of unity, and $\widetilde P_j(\zeta^{c})=\zeta^{cN}P_j(\zeta^{c})=0$, so its minimal polynomial $\Phi_{p^m}(w)=\Phi_p\bigl(w^{p^{m-1}}\bigr)$ divides $\widetilde P_j(w)$ in $\Z[w]$ (the divisor $\Phi_{p^m}$ is monic, so division of the integer polynomial $\widetilde P_j$ by it keeps quotient and remainder in $\Z[w]$), hence divides $P_j(w)$ in the Laurent ring $\Z[w,w^{-1}]$.
	The monomial factor $w^{N}$ shifts every exponent $\langle g,h\rangle$ by the constant $N$, which merely permutes the residues modulo $p^m$ and so leaves the multiplicity description in~(i) unchanged.
	
	Recall that each class $F_j$ is the intersection of $F$ with a single line parallel to $g_0$: it is the level set $\set{g\in F:\langle g,v\rangle=\text{const}}$, which is a line parallel to $g_0$ because $v\perp g_0$.
	Write $\ell_j$ for that line, so that $F_j=F\cap\ell_j=F_{\ell_j}$ and $P_j(w)=\sum_{g\in F_{\ell_j}}w^{\langle g,h\rangle}$.
	The divisibility $\Phi_{p^m}\mid P_j$ just established is then exactly the structural form~(i) for the line $\ell_j$.
	Applying Lemma~\ref{lemma:prime power vanishing sums} to the multiset $\set{\langle g,h\rangle : g\in F_{\ell_j}}$, the divisibility $\Phi_{p^m}\mid P_j$ is equivalent to the multiplicity function $r\mapsto\abs{\set{g\in F_{\ell_j}:\langle g,h\rangle\equiv r\pmod{p^m}}}$ being constant on every coset of the order-$p$ subgroup $p^{m-1}\Z/p^m\Z$ (the equivalent description in~(i)) and it yields $p\mid\abs{F_{\ell_j}}$, which is~(ii).
	Finally, the integer $m$ was fixed (as the $p$-adic valuation of $d$) before the grouping into the classes $F_j$, so it is the same for every line $\ell$ parallel to $g_0$.
	The classes $F_1,\dots,F_l$ exhaust the nonempty sections. For a line $\ell$ parallel to $g_0$ with $F_\ell=\varnothing$ the polynomial in~(i) is $0$ (divisible by $\Phi_{p^m}$ trivially) and $\abs{F\cap\ell}=0$, so~(i) and~(ii) hold trivially there too. Thus the conclusions cover every line parallel to $g_0$ with the single $m$.
\end{proof}

We will make use of the following theorem, which can be proved using the spectral theorem along with the dilation lemma coupled with an averaging argument.

\begin{theorem}[{\cite[Lemma~3.2]{bhattacharya_tilings}}]
	\label{theorem:bhattacharya support}
	For the spectral measure $\nu$ of a tiling there is a finite set $\Delta\subseteq\Z^2\setminus\set 0$ with
	$$
		\supp(\nu)\subseteq\bigcup_{g\in\Delta}\ker\chi_g.
	$$
	We take $\Delta$ of minimum size.
\end{theorem}

By the minimality of $\abs\Delta$, the elements of $\Delta$ are pairwise linearly independent: two parallel ones, $n_0h$ and $n_1h$ with $h\in\Z^2$ primitive, could be replaced by their common multiple $\operatorname{lcm}(n_0,n_1)\,h$, whose kernel contains both, leaving a smaller set with the same property.

The proof of Theorem~\ref{theorem:p squared one periodic} below splits into two cases, according to how many directions of the minimal spectral support $\Delta$ carry infinitely many solutions of the vanishing equations~\eqref{equation:vanishing equations}: either at least two directions do, and then $F$ is a product cluster and Proposition~\ref{proposition:product periodicity} applies, or at most one does, and then the spectral measure is so concentrated that $f$ is forced to be $1$-periodic.
Before stating the theorem we record an elementary fact that will be used. 

\begin{lemma}[Full affine span of a product cluster]
	\label{lemma:product full span}
	Let $A,B\subseteq\Z$ be finite nonempty sets and $F=A\times B\subseteq\Z^2$. Then $F$ has full
	affine span if and only if $\gcd(A-A)=\gcd(B-B)=1$.
\end{lemma}
\begin{proof}
	Since $A$ and $B$ are nonempty, $0\in A-A$ and $0\in B-B$. The difference set factors as
	\[
		F-F=(A\times B)-(A\times B)=(A-A)\times(B-B),
	\]
	so the subgroup it generates contains every $(u,0)$ with $u\in A-A$ (take the second coordinate
	$0\in B-B$) and every $(0,v)$ with $v\in B-B$. Hence
	\[
		\langle F-F\rangle\supseteq\langle A-A\rangle\times\{0\}+\{0\}\times\langle B-B\rangle
		=\langle A-A\rangle\times\langle B-B\rangle ,
	\]
	and the reverse inclusion is immediate, every element of $F-F$ already lying in
	$\langle A-A\rangle\times\langle B-B\rangle$. Therefore
	\[
		\langle F-F\rangle=\langle A-A\rangle\times\langle B-B\rangle
		=\gcd(A-A)\,\Z\ \times\ \gcd(B-B)\,\Z ,
	\]
	using that a subgroup of $\Z$ generated by a set of integers is $\gcd(\cdot)\,\Z$. By
	Definition~\ref{definition:affine span} full affine span means $\langle F-F\rangle=\Z^2$, which holds
	if and only if $\gcd(A-A)\,\Z=\Z$ and $\gcd(B-B)\,\Z=\Z$, i.e.\ $\gcd(A-A)=\gcd(B-B)=1$.
\end{proof}

The theorem below strengthens \cite[Theorem~5.10]{khetan_order2} from a partition into at most two $1$-periodic tilings to a single $1$-periodic tiling in the orbit closure.
The two proofs are compared in \S\ref{section:introduction}.

\begin{theorem}[Orbit-closure $1$-periodicity for prime-squared clusters]
	\label{theorem:p squared one periodic}
	Let $F\subseteq\Z^2$ be an exact cluster of full affine span with $\abs F=p^2$, $p$ prime, and let $T$ be an $F$-tiling.
	Then the orbit closure $\overline{\Z^2\cdot T}$ contains a $1$-periodic $F$-tiling.
	In particular Conjecture~\ref{conjecture:orbit closure one periodicity} holds whenever $\abs F=p^2$.
\end{theorem}
\begin{proof}
	Recall the setup of \S\ref{subsection:dynamical formulation}: the orbit closure $X=\overline{\Z^2\cdot T}$ carries a $\Z^2$-ergodic probability measure $\mu$, and from it we form $A=\set{x\in X:x(0,0)=1}$, $f=1_A$, the cyclic subspace $\mc H$, its spectral measure $\nu$, and the intertwining isomorphism $\theta\colon\mc H\to L^2(\T^2,\nu)$.

	The argument is a case analysis on the number of elements of $\Delta$ carrying infinitely many solutions of the vanishing equations~\eqref{equation:vanishing equations}.
	If two distinct directions do, then $F$ has divisible sections in two transverse directions and is therefore a product cluster $A'\times B'$ of prime-cardinality complete residue systems, to which Proposition~\ref{proposition:product periodicity} applies and yields an axis period of $T$.
	If at most one direction does, the spectral measure is concentrated enough to force $f$ itself to be $1$-periodic.
	Below are the details.

	Translating $F$ we may without loss of generality assume that $0\in F$.
	By Theorem~\ref{theorem:bhattacharya support} there is a minimal finite set $\Delta\subseteq\Z^2\setminus\set 0$ with $\supp(\nu)\subseteq\bigcup_{g\in\Delta}\ker\chi_g$.
	We distinguish two cases, according to how many directions of $\Delta$ carry infinitely many
	solutions of the vanishing equations~\eqref{equation:vanishing equations}.

	\begin{enumerate}[(1)]
		\item \emph{Suppose distinct $g_0,g_1\in\Delta$ each have infinitely
				many points of $\ker\chi_{g_i}$ satisfying \eqref{equation:vanishing equations}.}
			By part~(ii) of Lemma~\ref{lemma:divisible sections}, every line parallel to $g_0$ and every line parallel to $g_1$ meets $F$ in a multiple of $p$.
				As $g_0,g_1$ point in distinct directions (being distinct members of $\Delta$), Lemma~\ref{lemma:product structure}(i) shows that, with $u_0,u_1$ primitive direction vectors of $g_0,g_1$, the cluster is a product in the basis $(u_0,u_1)$,
			\[
				F=\set{au_0+bu_1:\ a\in A',\ b\in B'},\qquad A',B'\subseteq\Z,\ \abs{A'}=\abs{B'}=p.
			\]
			Since $F$ has full affine span, part~(iii) of that lemma makes $(u_0,u_1)$ a $\Z$-basis, with $\abs{\det(u_0,u_1)}=1$ and $\gcd(A'-A')=\gcd(B'-B')=1$, and in the unimodular coordinates $au_0+bu_1\mapsto(a,b)$ the cluster is the product $F=A'\times B'\subseteq\Z^2$.
			We claim $A'$ and $B'$ are complete residue systems modulo $p$.
			Apply Lemma~\ref{lemma:divisible sections} in the direction $g_0$ and let $m\ge1$ be the integer it provides.
			Write $q=\langle u_0,u_0\rangle=\abs{u_0}^2$ for the self-pairing of the primitive direction $h=u_0$ used in Lemma~\ref{lemma:divisible sections}.
			
			The idea is to convert the section divisibility of Lemma~\ref{lemma:divisible sections} into a divisibility for the one-variable mask polynomial $\widetilde A(x)=\sum_{a\in A'}x^{a}$, and then to apply Lemma~\ref{lemma:prime power vanishing sums} to conclude that $A'$ is a complete residue system modulo $p$.

			Fix $b\in B'$.
			The corresponding line parallel to $g_0$ meets $F$ in the section $F_\ell=\set{au_0+bu_1:a\in A'}$.
			Lemma~\ref{lemma:divisible sections}(i), applied in the direction $g_0$ with $h=u_0$, concerns the Laurent polynomial $\sum_{g\in F_\ell}w^{\langle g,h\rangle}$, in which the point $g\in F_\ell$ contributes the monomial $w^{\langle g,h\rangle}$ with exponent $\langle g,h\rangle$.
			For $g=au_0+bu_1$ this exponent is
			\[
				\langle g,h\rangle=\langle au_0+bu_1,u_0\rangle=qa+b\langle u_1,u_0\rangle,
			\]
			an affine function of $a$ with leading coefficient $q$ and a constant term $b\langle u_1,u_0\rangle$ independent of $a$.
			Summing $w^{\langle g,h\rangle}$ over the section therefore factors out that constant term:
			\[
				\sum_{g\in F_\ell}w^{\langle g,h\rangle}
				=w^{b\langle u_1,u_0\rangle}\sum_{a\in A'}w^{qa}
				=w^{b\langle u_1,u_0\rangle}\,\widetilde A\bigl(w^{q}\bigr).
			\]
			The prefactor $w^{b\langle u_1,u_0\rangle}$ is a unit in the Laurent ring, so the divisibility $\Phi_{p^m}(w)\mid\sum_{g\in F_\ell}w^{\langle g,h\rangle}$ of Lemma~\ref{lemma:divisible sections}(i) is the same as $\Phi_{p^m}(w)\mid\widetilde A(w^{q})$.
			With $\zeta=e^{2\pi i/p^m}$ the primitive $p^m$-th root of unity, this says $\widetilde A(\zeta^{q})=0$.

			The root $\zeta^{q}$ is again a prime-power root of unity: raising a primitive $p^m$-th root to the power $q$ yields a root of order $p^{m'}$, where $m'=m-v_p(q)$ and $v_p$ denotes the $p$-adic valuation.
			Here $m'\ge1$, for if $m'=0$ then $\zeta^{q}=1$ and $\widetilde A(\zeta^{q})=\abs{A'}=p\neq0$, contradicting $\widetilde A(\zeta^{q})=0$.
			Thus $\zeta^{q}$ is a primitive $p^{m'}$-th root of unity and a root of $\widetilde A$, so Lemma~\ref{lemma:prime power vanishing sums}, applied to the multiset $A'$, shows that the multiplicity function
			\[
				r\mapsto\abs{\set{a\in A':a\equiv r\pmod{p^{m'}}}}
			\]
			is constant on every coset of the order-$p$ subgroup $p^{m'-1}\Z/p^{m'}\Z$ of $\Z/p^{m'}\Z$.

			It remains to use the cardinality $\abs{A'}=p$.
			Each coset of the order-$p$ subgroup consists of $p$ residues modulo $p^{m'}$, on which the multiplicity takes a common value. If that value is $k$, the coset accounts for $pk$ elements of $A'$.
			Since the cosets partition the residues and $\abs{A'}=p$, exactly one coset has $k=1$ and every other has $k=0$.
			So $A'$ meets a single coset, once in each of its $p$ residues:
			\[
				A'\equiv\set{c+jp^{m'-1}:j=0,\dots,p-1}\pmod{p^{m'}}.
			\]
			Every difference of two elements of $A'$ is then a multiple of $p^{m'-1}$ modulo $p^{m'}$, so $p^{m'-1}\mid\gcd(A'-A')$.
			As $\gcd(A'-A')=1$ by full affine span, this forces $m'=1$.
			At $m'=1$ the modulus is $p$ and the order-$p$ subgroup is all of $\Z/p\Z$, so the single coset is every residue and $A'$ is a complete residue system modulo $p$.
			The identical argument in the direction $g_1$ shows $B'$ is a complete residue system modulo $p$.
			Let $M=(u_0\mid u_1)\in\mathrm{GL}_2(\Z)$ be the matrix with columns $u_0,u_1$, so that the coordinate map $a\,u_0+b\,u_1\mapsto(a,b)$ above is $M^{-1}$.
			It carries $F$ to the axis-aligned product $A'\times B'$ and $T$ to the $(A'\times B')$-tiling $M^{-1}T$, with $A',B'$ complete residue systems modulo $p$ and $\gcd(A'-A')=\gcd(B'-B')=1$.
			By Proposition~\ref{proposition:product periodicity}, $M^{-1}T$ has a coordinate period $(0,p)$ or $(p,0)$.
			Since $M\in\mathrm{GL}_2(\Z)$ maps periods of $M^{-1}T$ bijectively to periods of $T$, the tiling $T$ itself is $1$-periodic (and hence has a $1$-periodic point in its orbit closure).

		\item \emph{At most one direction carries infinitely many solutions.} If exactly one direction of $\Delta$ carries infinitely many solutions of~\eqref{equation:vanishing equations}, let $g_0\in\Delta$ be that direction.
			If none does, let $g_0$ be any element of $\Delta$, which is nonempty since $\supp(\nu)\neq\varnothing$ and $\supp(\nu)\subseteq\bigcup_{g\in\Delta}\ker\chi_g$.
			In either case, for every other $h\in\Delta\setminus\set{g_0}$ only finitely many points of $\ker\chi_h$ satisfy~\eqref{equation:vanishing equations} (when no direction is exceptional this holds for $g_0$ as well, which only helps).

			By the dilation lemma~\ref{theorem:dilation lemma}, every $x\in X$ is an $\alpha F$-tiling for $\alpha$ coprime to $p$, which is the identity
			\[
				\sum_{g\in F}(\alpha g)\cdot f=\mathbf 1_X
			\]
			in $L^2(X,\mu)$, where $\mathbf 1_X$ is the all-ones function and $(\alpha g)\cdot f$ is the translate of $f$ under the shift by $\alpha g$.
Now $\mathbf 1_X$ is invariant under every $g\in\Z^2$, and the only functions in $L^2(\T^2,\nu)$ invariant under all multiplications by $\chi_g$ are the multiples of $\mathbf 1_{\set 0}$, the indicator of the origin $0\in\T^2$.
Applying $\theta$ therefore gives, for some constant $\kappa$,
			\[
				\sum_{g\in F}\chi_{\alpha g}=\kappa\,\mathbf 1_{\set 0}\quad\text{in } L^2(\T^2,\nu),
				\qquad \alpha \text{ coprime to } p.
			\]
			For a function $\psi\colon\T^2\to\C$ write
			\[
				Z(\psi)=\set{\xi\in\T^2:\psi(\xi)=0}
			\]
			for its \define{zero set}.
			Each $\sum_{g\in F}\chi_{\alpha g}$ is a finite sum of characters, hence a continuous function on $\T^2$, so $Z\bigl(\sum_{g\in F}\chi_{\alpha g}\bigr)$ is a closed set.
			The equality above holds in $L^2(\T^2,\nu)$ and $\mathbf 1_{\set 0}$ vanishes off the origin, so $\sum_{g\in F}\chi_{\alpha g}=0$ holds $\nu$-almost everywhere on $\T^2\setminus\set0$.
			As the zero set is closed this gives
			\[
				\supp(\nu)
				\subseteq
				\lrb{
					\bigcap_{\gcd(\alpha,p)=1} Z\lrp{\sum_{g\in F}\chi_{\alpha g}}
					}\cup\set 0 .
			\]
			By the case hypothesis, for each $h\in\Delta\setminus\set{g_0}$ only finitely many points of $\ker\chi_h$ satisfy the vanishing equations~\eqref{equation:vanishing equations}, that is, the set
			$$
				S_h
				=\lrb{
					\bigcap_{\gcd(\alpha, p)=1}
						Z\lrp{\sum_{g\in F}\chi_{\alpha g}}
					}
					\cap\ker\chi_h
			$$
			is finite.
			Moreover, every support point off the origin satisfies all of \eqref{equation:vanishing equations}: by the inclusion above 
			$$
				\supp(\nu)\setminus\set 0\subseteq\bigcap_{\gcd(\alpha,p)=1}Z\lrp{\sum_{g\in F}\chi_{\alpha g}},
			$$
			so $\supp(\nu)\cap\ker\chi_h\subseteq S_h\cup\set 0$ for each such $h$.
			Combining with $\supp(\nu)\subseteq\bigcup_{g\in\Delta}\ker\chi_g$, the finite set $S=\lrp{\bigcup_{h\in\Delta\setminus\set{g_0}}S_h}\cup\set 0$ satisfies
			\[
				\supp(\nu)\subseteq\ker\chi_{g_0}\cup S .
			\]
			We now fix $S$ economically. Among all finite sets $S\subseteq\T^2$ and all positive integers $k$ for which $\supp(\nu)\subseteq\ker\chi_{kg_0}\cup S$, choose a pair making $\abs S$ as small as possible, and relabel $g_0$ to be the corresponding multiple, so that
			\[
				\supp(\nu)\subseteq\ker\chi_{g_0}\cup S
			\]
			holds with $S$ of least size.
			Discarding any points of $S$ lying in $\ker\chi_{g_0}$ leaves the inclusion intact without increasing $\abs S$, so we may also take $S$ disjoint from $\ker\chi_{g_0}$.
			Three properties follow from this minimal choice.
			\begin{enumerate}[(i)]
				\item \emph{Each $s\in S$ lies in $\supp(\nu)$.}
					Otherwise we would have
					\[
						\supp(\nu)\subseteq\ker\chi_{g_0}\cup(S\setminus\set s)
					\]
					(using $s\notin\ker\chi_{g_0}$), contradicting the minimality of $\abs S$.
				\item \emph{Each $s\in S$ has positive $\nu$-mass.}
					Since $S$ is finite and disjoint from the closed set $\ker\chi_{g_0}$, the point $s$ is isolated in $\ker\chi_{g_0}\cup S$, hence in $\supp(\nu)$, that is, some neighbourhood $U$ of $s$ meets $\supp(\nu)$ only at $s$.
					Then $\nu(\set s)=\nu(U)>0$, the inequality holding because $s\in\supp(\nu)$ by~(i).
				\item \emph{$\chi_{g_0}(s)$ is irrational, i.e.\ not a root of unity, for each $s\in S$.}
					Were $\chi_{g_0}(s)$ a root of unity, of order $k'$ say, then $\chi_{k'g_0}(s)=\chi_{g_0}(s)^{k'}=1$, that is $s\in\ker\chi_{k'g_0}$.
					Since $\ker\chi_{g_0}\subseteq\ker\chi_{k'g_0}$, this would give $\supp(\nu)\subseteq\ker\chi_{k'g_0}\cup(S\setminus\set s)$, a pair $(k'g_0,\,S\setminus\set s)$ with a strictly smaller set, contradicting minimality.
			\end{enumerate}

			\noindent
			We now show that $A$ is $1$-periodic, whether or not $S$ is empty. When $S=\varnothing$ the classes below are absent and the conclusion $g_0\cdot f=f$ is immediate.
			Partition $S$ by the value of $\chi_{g_0}$ into classes $E\in\mathcal E$ with representatives $p_E$.
			Then 
			$$
				\mathbf 1=\mathbf 1_{\ker\chi_{g_0}}+\sum_{E\in \mc E} \mathbf 1_E
			$$
			in $L^2(\T^2,\nu)$.
			Multiplying by $\chi_{ng_0}=\chi_{g_0}^{\,n}$, we get
			\[
				\chi_{ng_0}=\mathbf 1_{\ker\chi_{g_0}}+\sum_{E\in \mc E}\chi_{g_0}(p_E)^{n}\,\mathbf 1_E .
			\]
			Define $\varphi_E=\norm f_2\,\theta^{-1}(\mathbf 1_E)$ for each $E\in\mc E$, and let $f^{g_0}$ be the orthogonal projection of $f$ onto the $g_0$-invariant vectors of the cyclic subspace $\mc H$.
			Note that $f^{g_0}=\norm f_2\,\theta^{-1}(\mathbf 1_{\ker\chi_{g_0}})$.\footnote{
				Under $\theta$ the translation $g_0\cdot{}$ becomes multiplication by $\chi_{g_0}$, so a vector $h\in\mc H$ is $g_0$-invariant precisely when $\chi_{g_0}\,\theta(h)=\theta(h)$, that is, when $\theta(h)$ vanishes $\nu$-almost everywhere off $\ker\chi_{g_0}=\set{\xi:\chi_{g_0}(\xi)=1}$.
				Hence $\theta$ maps the $g_0$-invariant subspace onto $\mathbf 1_{\ker\chi_{g_0}}\,L^2(\T^2,\nu)$, and the orthogonal projection onto it corresponds under $\theta$ to multiplication by $\mathbf 1_{\ker\chi_{g_0}}$.
				Applying this to $f$ and using $\theta(f)=\norm f_2\,\mathbf 1$ gives $\theta(f^{g_0})=\norm f_2\,\mathbf 1_{\ker\chi_{g_0}}$, which is the stated relation.
			}
			Apply the inverse isometry $\theta^{-1}$ to the identity above and multiply through by $\norm f_2$, treating each summand separately.
			We get
			\[
				(ng_0)\cdot f=f^{g_0}+\sum_{E\in\mc E}\chi_{g_0}(p_E)^{n}\,\varphi_E\qquad\text{for all }n\in\Z .
			\]
			Specialising to $n=0$, where $\chi_0=\mathbf 1$ and $(0\cdot g_0)\cdot f=f$, gives $f=f^{g_0}+\sum_{E\in\mc E}\varphi_E$.
			Subtracting this from the previous identity cancels the $n$-independent term $f^{g_0}$ and leaves
			\[
				(ng_0)\cdot f - f
				=\sum_{E\in \mc E}
				\bigl(\chi_{g_0}(p_E)^{n}-1\bigr)\varphi_E .
			\]
			On a $\mu$-full set $Y\subseteq X$ this is a pointwise identity, and since $f$ and $(ng_0)\cdot f$ are $\set{0,1}$-valued their difference is an integer, so for every $y\in Y$ and $n\ge 0$,
			\[
				\sum_{E\in \mc E}\bigl(\chi_{g_0}(p_E)^{n}-1\bigr)\varphi_E(y)\in\Z .
			\]
			As each $\chi_{g_0}(p_E)$ is irrational, Lemma~\ref{lemma:rational independence vanishing} (whose only hypothesis is the irrationality of each $\gamma$, multiplicative relations among them being handled in its proof) forces this expression to be $0$ for all $n$.
			Taking $n=1$ gives $[g_0\cdot f](y)=f(y)$ for $y\in Y$, hence $g_0\cdot f=f$ in $L^2(X, \mu)$.
			Since $f=\mathbf 1_A$, this says $g_0\cdot A=A$ up to a $\mu$-null set, so $A$ is $1$-periodic with period $g_0$.
			By Theorem~\ref{theorem:bhattacharya correspondence}, $\mu$-almost every point of $X$ is a $1$-periodic $F$-tiling.
	\end{enumerate}

	In both cases $\overline{\Z^2\cdot T}$ contains a $1$-periodic $F$-tiling.
\end{proof}

\subsection{The threshold, and the open case \texorpdfstring{$\abs F=6$}{|F|=6}}
\label{subsection:threshold open}

How small can a cluster be before orbit-closure $1$-periodicity fails?
For every cardinality below $8$ except $6$, it does not fail.
When $\abs F$ is prime, Szegedy's theorem~\cite{szegedy_algorithms_to_tile} says that every $F$-tiling is already $1$-periodic, which covers $\abs F=2,3,5,7$.
The case $\abs F=1$ is trivial, and $\abs F=4=2^2$ is Theorem~\ref{theorem:p squared one periodic} above.
The eight-cell cluster of Section~\ref{section:counterexample}, on the other hand, fails: orbit-closure $1$-periodicity does not hold when $\abs F=8$.
So the smallest cardinality at which it can fail is either $6$ or $8$.
Only the case $\abs F=6$ is left. We cannot decide it, and we leave it open.

\appendix

\section{Proof of the prime-power vanishing criterion}
\label{appendix:vanishing sums}

Lemma~\ref{lemma:prime power vanishing sums} is the prime-power case of the R\'edei--de Bruijn--Schoenberg theorem, quoted in the main text from \cite[Theorem~2.2]{lam_leung_vanishing_sums}, whose proof is written in the language of integral group rings.
We claim no new argument.
The sole purpose of this appendix is to rephrase that proof in the language of polynomial rings, which keeps the paper self-contained and may be more transparent to a reader not accustomed to group-ring manipulations.
The translation is the ring isomorphism
\[
	\Z G\ \xrightarrow{\ \sim\ }\ \Z[z]/(z^{p^m}-1),\qquad g\mapsto z,
\]
where $G=\lrab g$ is cyclic of order $p^m$: Lam and Leung's kernel ideal $\Z G\cdot\Phi_{p^m}$ and its generator $\sigma(P_1)=\Phi_{p^m}$ become divisibility by the cyclotomic polynomial, and their ``constant on the cosets of the order-$p$ subgroup'' is condition~(c) verbatim.
The one computation, the equivalence of the divisibility condition~(b) and the coset-constancy condition~(c), we carry out by expanding a polynomial in powers of $z^{p^{m-1}}$ and checking divisibility one coset at a time.
Three classical facts are used: the identity $\Phi_{p^m}(z)=\Phi_p\bigl(z^{p^{m-1}}\bigr)$, that $\Phi_{p^m}$ is the minimal polynomial over $\Q$ of a primitive $p^m$-th root of unity, and Gauss's lemma on divisibility in $\Z[z]$.
Beyond these, only unique factorization in $\Z[z]$ is invoked, through Euclid's lemma.

Throughout we write
\[
	N=p^m,\qquad d=p^{m-1}
\]
so that
\[
	N=pd\quad\text{and}\quad
	\Phi_{p^m}(z)=\Phi_p\bigl(z^{d}\bigr)=1+z^{d}+z^{2d}+\cdots+z^{(p-1)d}.
\]
For a residue $k$ we write $c_k=\abs{\set{a\in S:\ a\equiv k\pmod N}}$ for the multiplicity function of condition~(c), the count taken with multiplicity, so that an element of $S$ occurring $j$ times in the multiset contributes $j$.
If $k\equiv k'\pmod N$ then $c_k=c_{k'}$.
Thus $c_k$ depends only on the class $k+N\Z$, and $c$ descends to a well-defined function on $\Z/N\Z$, sending a class to the number of elements of $S$ in it.
This is what lets condition~(c) speak of $c$ being constant on a coset.

\begin{proof}[Proof of Lemma~\ref{lemma:prime power vanishing sums}]
	First we reduce to the case where $S$ consists only of non-negative integers.
	If $S=\varnothing$ then $P=0$, all three conditions hold, and $p\mid 0$.
	Assume henceforth that $S\neq\varnothing$.
	Translating the exponents, $S\mapsto S+t$, replaces $P(z)$ by $z^{t}P(z)$ and leaves all three conditions unchanged, as we check in turn.
	\begin{enumerate}[(a)]
		\item The value $P(\zeta)$ changes only by the nonzero factor $\zeta^{t}$, so the vanishing $P(\zeta)=0$ is preserved.
		\item When $S$ has negative elements $P$ is a Laurent polynomial, so condition~(b) is not yet a statement about $\Z[z]$ and must first be read through a clearing monomial: for any integer $T$ with $z^{T}P(z)\in\Z[z]$, it means $\Phi_{p^m}(z)\mid z^{T}P(z)$ in $\Z[z]$.
			Because $\Phi_{p^m}(0)=1$, the polynomial $\Phi_{p^m}$ is coprime to $z$ in $\Z[z]$, and hence to every power of $z$.
			By Euclid's lemma in the unique factorization domain $\Z[z]$ (if $c\mid ab$ and $c$ is coprime to $a$, then $c\mid b$), multiplying the cleared polynomial $z^{T}P$ by a further power of $z$ does not change whether $\Phi_{p^m}$ divides it.
			Any two clearings are related by exactly such a multiplication: if $T\le T'$ are two valid choices, then $z^{T'}P=z^{T'-T}\,\bigl(z^{T}P\bigr)$.
			Hence the reading is independent of $T$, and the translation $P\mapsto z^{t}P$, which merely shifts the clearing monomial, leaves condition~(b) unchanged.
			(Once $S\subseteq\set{0,1,2,\dots}$, so that $P(z)\in\Z[z]$, no clearing is needed and~(b) is literally $\Phi_{p^m}\mid P$ in $\Z[z]$.)
		\item The multiplicity function is merely translated by $t$, so its constancy on the cosets of the order-$p$ subgroup is preserved.
	\end{enumerate}
	Translation also leaves the cardinality $\abs S$ unchanged, so the final divisibility $p\mid\abs S$ may likewise be proved after translating.
	Taking $t=-\min S$, we may therefore assume outright that $S\subseteq\set{0,1,2,\dots}$, so that $P(z)\in\Z[z]$.

	\medskip \noindent
	\emph{Reduction modulo $z^{N}-1$.}
	Dividing by $z^{N}-1$ and collecting the exponents of $P$ by their residue modulo $N$ gives
	\[
		P(z)\equiv\bar P(z):=\sum_{k=0}^{N-1}c_k\,z^{k}\pmod{z^{N}-1},
	\]
	where $c_k$ is the multiplicity above, because reducing $z^{a}$ modulo $z^{N}-1$ replaces it by $z^{a\bmod N}$.
	The geometric-sum identity $(w-1)\bigl(1+w+\cdots+w^{p-1}\bigr)=w^{p}-1$, evaluated at $w=z^{d}$, gives the factorization $z^{N}-1=\bigl(z^{d}-1\bigr)\Phi_{p^m}(z)$, so $\Phi_{p^m}(z)\mid z^{N}-1$ and hence $\Phi_{p^m}$ divides $P$ if and only if it divides $\bar P$.
	Likewise, $\zeta^{N}=1$ gives $P(\zeta)=\bar P(\zeta)$.
	We may thus work with $\bar P$, a polynomial of degree less than $N$, in place of $P$ throughout.

	\medskip \noindent
	\emph{Equivalence of \textup{(a)} and \textup{(b)}.}
	If $\Phi_{p^m}\mid\bar P$ in $\Z[z]$ then $\bar P(\zeta)=0$, because $\Phi_{p^m}(\zeta)=0$. This gives~(b)$\Rightarrow$(a).
	Conversely, suppose $\bar P(\zeta)=0$.
	As $\Phi_{p^m}$ is the minimal polynomial of $\zeta$ over $\Q$, it divides $\bar P$ in $\Q[z]$.
	Both $\Phi_{p^m}$ and $\bar P$ lie in $\Z[z]$ and $\Phi_{p^m}$ is monic, so the quotient already lies in $\Z[z]$ by Gauss's lemma, giving $\Phi_{p^m}\mid\bar P$ in $\Z[z]$.
	This is~(a)$\Rightarrow$(b).

	\medskip \noindent
	\emph{Equivalence of \textup{(b)} and \textup{(c)}.}
	Every $k\in\set{0,1,\dots,N-1}$ has a unique representation $k=r+jd$ with $0\le r<d$ and $0\le j<p$, coming from division by $d$.
	Grouping the terms of $\bar P$ accordingly expands it in powers of $z^{d}$:
	\begin{equation}
		\label{equation:base-d expansion}
		\bar P(z)
		=\sum_{r=0}^{d-1}z^{r}\,B_r\bigl(z^{d}\bigr),
			\qquad\text{where}\qquad
		B_r(w)
		=\sum_{j=0}^{p-1}c_{r+jd}\,w^{j}.
	\end{equation}
	Thus $B_r$ has degree at most $p-1$.
	Since \eqref{equation:base-d expansion} merely regroups the expansion $\bar P(z)=\sum_{k}c_{k}\,z^{k}$, the coefficient of $B_r$ on $w^{j}$ is $c_{r+jd}$, the value at the residue $r+jd$ of the multiplicity function of condition~(c).
	In the summand $z^{r}B_r\bigl(z^{d}\bigr)$ the term $c_{r+jd}\,w^{j}$ of $B_r$ becomes the monomial $c_{r+jd}\,z^{r+jd}$, so the exponents of $z$ appearing in this summand are $r,r+d,\dots,r+(p-1)d$.
	We now identify these $p$ exponents as a single coset modulo $N$.
	For an integer $a$ write $\bar a=a+p^m\Z$ for its class in $\Z/N\Z$.
	The order-$p$ subgroup of $\Z/N\Z$ is
	\[
		K
		= p^{m-1}\Z/p^m\Z
		= \set{\,\bar0,\ \bar d,\ \overline{2d},\ \dots,\ \overline{(p-1)d}\,},
	\]
	the cyclic subgroup generated by $\bar d$ (recall $d=p^{m-1}$).
	It has order $p$ because $pd=N$.
	Since $\overline{r+jd}=\bar r+\overline{jd}$ and $\overline{jd}$ ranges over all of $K$ as $j$ runs from $0$ to $p-1$, the classes $\overline{r+jd}$ are exactly the $p$ elements of the coset $\bar r+K$.
	In summary, the coefficients $c_r,c_{r+d},\dots,c_{r+(p-1)d}$ of $B_r$ are the values of the multiplicity function at the $p$ residues $r,r+d,\dots,r+(p-1)d$, whose classes make up the coset $\bar r+K$.

	We claim that $\Phi_{p^m}(z)=\Phi_p(z^{d})$ divides $\bar P(z)$ in $\Z[z]$ if and only if $\Phi_p(w)$ divides each $B_r(w)$ in $\Z[w]$.
	Multiplication by $\Phi_p(z^{d})$ preserves the residue of an exponent modulo $d$, since it involves only exponents divisible by $d$, and this is what decouples the divisibility across residues.
	If $B_r(w)=\Phi_p(w) Q_r(w)$ for polynomials $Q_r(w)\in\Z[w]$, then~\eqref{equation:base-d expansion} gives
	\[
		\bar P(z)
		=\Phi_p\bigl(z^{d}\bigr)\sum_{r=0}^{d-1}z^{r}\,Q_r\bigl(z^{d}\bigr),
	\]
	so $\Phi_{p^m}\mid\bar P$.
	Conversely, suppose $\bar P(z)=\Phi_p(z^{d})\,H(z)$ with $H(z)\in\Z[z]$.
	Expanding $H(z)=\sum_{s=0}^{d-1}z^{s}H_s(z^{d})$ in powers of $z^{d}$ as in~\eqref{equation:base-d expansion}, the summand $z^{s}\Phi_p(z^{d})H_s(z^{d})$ carries only exponents congruent to $s$ modulo $d$.
	Comparing the parts of $\bar P$ and of $\Phi_p(z^{d})H(z)$ whose exponents are congruent to $r$ modulo $d$ therefore gives
	\[
		z^{r}B_r\bigl(z^{d}\bigr)
		=z^{r}\Phi_p\bigl(z^{d}\bigr)H_r\bigl(z^{d}\bigr),
	\]
	whence $B_r(w)=\Phi_p(w)H_r(w)$ after cancelling $z^{r}$ and setting $w=z^{d}$, the substitution $w\mapsto z^{d}$ being an injection of $\Z[w]$ into $\Z[z]$.
	This proves the claim.

	It remains to show that divisibility of each $B_r$ by $\Phi_p$ amounts to coset-constancy.
	Since $\deg B_r\le p-1=\deg\Phi_p$, the divisibility $\Phi_p\mid B_r$ holds if and only if $B_r$ is a constant multiple of $\Phi_p$, say $B_r=c^{\ast}_r\,\Phi_p$ with $c^{\ast}_r\in\Z$: a nonzero quotient $Q_r$ in $B_r=\Phi_p\,Q_r$ has $\deg Q_r=\deg B_r-\deg\Phi_p\le 0$ and is therefore a nonzero integer, while $B_r=0$ corresponds to the constant $c^{\ast}_r=0$.
	Because $\Phi_p(w)=1+w+\cdots+w^{p-1}$ has all coefficients equal to $1$, this says exactly that $c_r=c_{r+d}=\cdots=c_{r+(p-1)d}$, all equal to $c^{\ast}_r$.
	As these $p$ residues represent the classes of the coset $\bar r+K$, and $c$ is a well-defined function on $\Z/N\Z$, this is precisely the statement that $c$ is constant on the coset $\bar r+K$.
	Letting $r$ range over $0,\dots,d-1$ covers every coset of the order-$p$ subgroup: each class in $\Z/N\Z$ is represented by some $k\in\set{0,1,\dots,N-1}$, and writing $k=r+jd$ as above places that class in the coset $\bar r+K$.
	So $\Phi_p\mid B_r$ for all $r$ is precisely condition~(c).
	Combined with the claim, this is the equivalence of~(b) and~(c).

	\smallskip \noindent
	\emph{The divisibility $p\mid\abs S$.}
	Suppose the equivalent conditions hold, so that $c_{r+jd}=c^{\ast}_r$ is independent of $j$.
	Summing the multiplicities over all residues,
	\[
		\abs S=\sum_{k=0}^{N-1}c_k=\sum_{r=0}^{d-1}\sum_{j=0}^{p-1}c_{r+jd}=\sum_{r=0}^{d-1}p\,c^{\ast}_r=p\sum_{r=0}^{d-1}c^{\ast}_r,
	\]
	which is divisible by $p$.
\end{proof}

\bibliographystyle{alpha}
\bibliography{OCC.bib}

\begin{thebibliography}{GHdMV18}

\bibitem[Bha16]{bhattacharya_tilings}
Siddhartha Bhattacharya.
\newblock Periodicity and decidability of tilings of {$\mathbb Z^2$}.
\newblock \texttt{arXiv:1602.05738}, 2016.

\bibitem[Cas00]{cassaigne_subword_complexity}
Julien Cassaigne.
\newblock Subword complexity and periodicity in two or more dimensions.
\newblock In Grzegorz Rozenberg and Wolfgang Thomas, editors, {\em Developments
  in Language Theory ({DLT} 1999)}, pages 14--21, River Edge, NJ, 2000. World
  Scientific.

\bibitem[CK15]{cyr_kra_nonexpansive}
Van Cyr and Bryna Kra.
\newblock Nonexpansive $\mathbb{Z}^2$-subdynamics and {N}ivat's conjecture.
\newblock {\em Transactions of the American Mathematical Society},
  367(9):6487--6537, 2015.

\bibitem[CM99]{coven_meyerowitz_tiling_integers}
Ethan~M. Coven and Aaron Meyerowitz.
\newblock Tiling the integers with translates of one finite set.
\newblock {\em Journal of Algebra}, 212(1):161--174, 1999.

\bibitem[EW11]{einsiedler_ward_ergodic_theory}
Manfred Einsiedler and Thomas Ward.
\newblock {\em Ergodic theory with a view towards number theory}, volume 259 of
  {\em Graduate Texts in Mathematics}.
\newblock Springer-Verlag London, Ltd., London, 2011.

\bibitem[GHdMV18]{grandjean_menibus_vanier_2018}
Anael Grandjean, Benjamin Hellouin~de Menibus, and Pascal Vanier.
\newblock Aperiodic points in {$\mathbb Z^2$}-subshifts.
\newblock In {\em 45th {I}nternational {C}olloquium on {A}utomata, {L}anguages,
  and {P}rogramming}, volume 107 of {\em LIPIcs. Leibniz Int. Proc. Inform.},
  pages Art. No. 128, 13. Schloss Dagstuhl. Leibniz-Zent. Inform., Wadern,
  2018.

\bibitem[GT21]{greenfeld_tao_2020}
Rachel Greenfeld and Terence Tao.
\newblock The structure of translational tilings in $\mathbb{Z}^d$.
\newblock {\em Discrete Anal.}, pages Paper No. 16, 28, 2021.
\newblock \texttt{arXiv:2010.03254}.

\bibitem[GT24]{greenfeld_tao_counterexample}
Rachel Greenfeld and Terence Tao.
\newblock A counterexample to the periodic tiling conjecture.
\newblock {\em Ann. of Math. (2)}, 200(1):301--363, 2024.
\newblock \texttt{arXiv:2211.15847}.

\bibitem[HK16]{horak_kim_algebraic_method}
Peter Horak and Dongryul Kim.
\newblock Algebraic method in tilings.
\newblock \texttt{arXiv:1603.00051}, 2016.

\bibitem[Kar19]{kari_low_complexity_survey}
Jarkko Kari.
\newblock Low-complexity tilings of the plane.
\newblock In {\em Descriptional complexity of formal systems}, volume 11612 of
  {\em Lecture Notes in Comput. Sci.}, pages 35--45. Springer, Cham, 2019.
\newblock Invited paper, DCFS 2019; \texttt{arXiv:1905.04183}.

\bibitem[Khe21a]{khetan_order2}
Abhishek Khetan.
\newblock On configurations of order 2.
\newblock \texttt{arXiv:2102.00803}, 2021.

\bibitem[Khe21b]{khetan_prime_squared}
Abhishek Khetan.
\newblock A periodicity result for tilings of $\mathbb{Z}^3$ by clusters of
  prime-squared cardinality.
\newblock \texttt{arXiv:2109.14179}, 2021.

\bibitem[KM23]{kari_mutot_low_comp}
Jarkko Kari and {\'E}tienne Moutot.
\newblock Decidability and periodicity of low complexity tilings.
\newblock {\em Theory Comput. Syst.}, 67(1):125--148, 2023.
\newblock Extended version of a paper in STACS 2020; \texttt{arXiv:1904.01267}.

\bibitem[KS20]{kari_szabados_alg_geom}
Jarkko Kari and Michal Szabados.
\newblock An algebraic geometric approach to {N}ivat's conjecture.
\newblock {\em Inform. and Comput.}, 271:104481, 25, 2020.

\bibitem[LL00]{lam_leung_vanishing_sums}
T.~Y. Lam and K.~H. Leung.
\newblock On vanishing sums of roots of unity.
\newblock {\em Journal of Algebra}, 224(1):91--109, 2000.

\bibitem[LL23]{laba_londner_three_primes}
Izabella \L{}aba and Itay Londner.
\newblock The {C}oven--{M}eyerowitz tiling conditions for 3 odd prime factors.
\newblock {\em Inventiones Mathematicae}, 232(1):365--470, 2023.

\bibitem[Lot02]{lothaire_words}
M.~Lothaire.
\newblock {\em Algebraic combinatorics on words}, volume~90 of {\em
  Encyclopedia of Mathematics and its Applications}.
\newblock Cambridge University Press, Cambridge, 2002.
\newblock Chapter 2: Sturmian words.

\bibitem[MH40]{morse_hedlund_1940}
Marston Morse and Gustav~A. Hedlund.
\newblock Symbolic dynamics {II}. {S}turmian trajectories.
\newblock {\em Amer. J. Math.}, 62(1):1--42, 1940.

\bibitem[Niv97]{nivat1997}
Maurice Nivat.
\newblock Invited talk at {ICALP} 1997.
\newblock 25th International Colloquium on Automata, Languages and Programming
  (EATCS keynote), 1997.

\bibitem[ST00]{sander_tijdeman_lattices}
J.~W. Sander and R.~Tijdeman.
\newblock The complexity of functions on lattices.
\newblock {\em Theoretical Computer Science}, 246(1--2):195--225, 2000.

\bibitem[Sze98]{szegedy_algorithms_to_tile}
Mario Szegedy.
\newblock Algorithms to tile the infinite grid with finite clusters.
\newblock In {\em 39th Annual Symposium on Foundations of Computer Science
  (FOCS 1998)}. IEEE, 1998.

\bibitem[Tij95]{tijdeman_decomposition_of}
R.~Tijdeman.
\newblock Decomposition of the integers as a direct sum of two subsets.
\newblock In {\em Number theory ({P}aris, 1992--1993)}, volume 215 of {\em
  London Math. Soc. Lecture Note Ser.}, pages 261--276. Cambridge Univ. Press,
  Cambridge, 1995.

\end{thebibliography}

\end{document}